\documentclass{amsart}

\usepackage{appendix}
\usepackage{graphicx}
\usepackage{amsmath}
\usepackage{mathrsfs}

\newtheorem{theorem}{Theorem}[section]
\newtheorem{lemma}[theorem]{Lemma}
\newtheorem{corollary}[theorem]{Corollary}
\newtheorem{proposition}[theorem]{Proposition}

\theoremstyle{definition}
\newtheorem{definition}[theorem]{Definition}
\newtheorem{hypothesis}[theorem]{Hypothesis}
\newtheorem{example}[theorem]{Example}

\theoremstyle{remark}
\newtheorem{remark}[theorem]{Remark}

\numberwithin{equation}{section}

\begin{document}

\title[Feynman-Kac formula on semi-Dirichlet forms]{Bivariate Revuz measures and the Feynman-Kac formula on semi-Dirichlet forms}

\author{Liping Li}
\address{School of Mathematical Sciences, Fudan University. 220 Handan Road, Shanghai 200433 China}
\email{lipingli10@fudan.edu.cn}

\author{Jiangang Ying}
\address{School of Mathematical Sciences, Fudan University. 220 Handan Road, Shanghai 200433 China}
\email{jgying@fudan.edu.cn}





\begin{abstract}
In this paper we shall first establish the theory of bivariate Revuz correspondence of positive additive functionals
under a semi-Dirichlet form which is associated with a right Markov process $X$ satisfying the sector condition but
without duality. We  extend most of the classical results about the bivariate Revuz measures under the duality assumptions
to the case of semi-Dirichlet forms. As the main results of this paper, we prove that
for any exact multiplicative functional $M$ of $X$, the subprocess $X^M$ of $X$ killed by $M$ also satisfies the sector condition and  we then characterize
the semi-Dirichlet form associated with $X^M$ by using the bivariate Revuz measure, which extends the classical Feynman-Kac formula.
\end{abstract}

\maketitle

\section{Introduction}

We shall briefly explain the title of this paper first.
The original Feynman-Kac formula is the characterization of the transition semigroup corresponding
to the classical Schr\"odinger equation. Hence any topic related to this is called a Feynman-Kac formula.
The essential point of Dirichlet form theory is the one-to-one correspondence between Markov processes
and Dirichlet forms. An decreasing multiplicative functional of a Markov process gives us a
subprocess and its transition semigroup, which corresponds to the generalized Schr\"odinger equation.
The Feynman-Kac formula means the characterization of Dirichlet form of the subprocess, if it is valid.

Another word in the title we need to explain is semi-Dirichlet form. The classical theory of Dirichlet
form, referring to \cite{FU},
is the energy form of a Markov process $X$ which is symmetric with respect to a $\sigma$-finite
measure $m$ on the state space $E$. This theory was extended to non-symmetric
Dirichlet form where a pair of Markov processes are dual with respect to $m$ and the bilinear form
corresponding to the infinitesimal generator satisfies so-called sector condition so that
theory of functional analysis can be used. For non-symmetric Dirichlet form,
 refer to \cite{MR}. More generally, the semi-Dirichlet form is the bilinear form of a Markov
process whose infinitesimal generator satisfies the sector condition with respect to a measure.
The big difference between Dirichlet form and semi-Dirichlet form
is that the measure $m$ is excessive for the associated process in former case, so
that the results in probabilistic potential theory may be used directly, and not for the later case.
For semi-Dirichlet forms, refer to \cite{PJF}.

The main purpose of this paper is to prove that any subprocess of a Markov process
associating with a semi-Dirichlet form satisfies the sector condition and to characterize the semi-Dirichlet
form of the subprocess. For the Feynman-Kac formula concerning decreasing continuous multiplicative functionals
in symmetric, non-symmetric and semi-Dirichlet form, refer to \cite{FU}, \cite{MR} and \cite{PJF}, respectively.
For one concerning decreasing multiplicative
functionals (non-local MF's), refer to \cite{JGY} and \cite{JGY3}.

We shall adopt the standard notation and terminology of right Markov processes in 
\cite{BG}, \cite{RGK4} and \cite{SMJ2}. The symbol
`:=' means a definition. 
Let $(E,\mathcal{B})$ be a metrizable Lusin space, and $m$ a $\sigma$-finite measure on $E$.  Let
$$X=(\Omega,\mathcal{M},\mathcal{M}_t,X_t,\theta_t,P^x)$$
be a right Markov process on $E\cup\{\Delta\}$, where
        $\Delta$ is the trap of $X$, with $(P_t)$ as its transition semigroup and
$\zeta$ as its  \emph{life time}. 
 Assume that $(P_t)_{t\geq 0}$ satisfies the following hypothesis.
\begin{hypothesis}\label{HYP3}
    The semigroup $(P_t)_{t\geq 0}$ acts as a strongly continuous contraction semigroup on $L^2(E,m)$.
\end{hypothesis}
Note that \text{Hypothesis \ref{HYP3}} is not trivial because $m$ may not be excessive.
 The infinitesimal generator of $(P_t)_{t\geq 0}$ is the densely defined operator $L$ given by
\[
  Lf:=\lim_{t\rightarrow 0}(P_tf-f)/t,
  \]
 with the domain $D(L)$ being the class of $f\in L^2(E,m)$ for which the indicated limit exists in the strong sense in $L^2(E,m)$. The process $X$ is said to satisfy the \emph{sector condition} if the following hypothesis holds.
 \begin{hypothesis}[Sector condition]\label{HYP1}
  There is a constant $K\geq 1$ such that the bilinear form
  \[
    \mathcal{E}(f, g):=(f, -Lg),\quad f,g\in D(L)
  \]
 satisfies
  \[
    |\mathcal{E}(f,g)|\leq K\cdot (f,(I-L)f)^{\frac{1}{2}}\cdot (g,(I-L)g)^{\frac{1}{2}}
  \]
 for any $f,g\in D(L)$, where $I$ is the identity.
  \end{hypothesis}
Under the sector condition $(\mathcal{E},D(L))$ can be extended to a semi-Dirichlet form (see the appendix) denoted by $(\mathcal{E},\mathcal{F})$ and $D(L)$ is $\tilde{\mathcal{E}}_1$-dense in $\mathcal{F}$. Moreover in  \cite{PJF} the author proved that under a mild assumption ($E$ should be a metrizable co-Souslin space, see HYPOTHESIS2.1 of  \cite{PJF}) a right Markov process for which the sector condition holds is necessarily $m$-\emph{standard}, $m$-\emph{special} and $m$-\emph{tight}. In particular its associated semi-Dirichlet form $(\mathcal{E},\mathcal{F})$ is \emph{quasi-regular} and $X$ is properly associated with $(\mathcal{E},\mathcal{F})$. It is well known that (see  \cite{MOR}) similar to the classical case, every quasi-regular semi-Dirichlet form on $L^2(E,m)$ always corresponds to an $m$-tight special standard process. However the state space $E$ of a right Markov process is usually assumed
to be a metrizable Radon space. The required assumption, say HYPOTHESIS2.1 of  \cite{PJF}, can be replaced by the following hypothesis. Note that here we assume that $E$ is separable whereas  \cite{PJF} does not.

  \begin{hypothesis}\label{HYP5}
    The space $E$ is a separable metrizable Radon space, and there is an increasing sequence $\{K_n\}_{n\geq 1}$ of compact subsets of $E$ such that
    \begin{equation}\label{MT}
    P^m(\lim_{n\rightarrow \infty} T_{E\setminus K_n}<\zeta )=0,
    \end{equation}
     where $T_B:=\inf\{t>0:X_t\in B\}$ for any Borel subset $B$ of $E$.
  \end{hypothesis}

 Throughout this paper we always assume that $X$ is a right Markov process on $E$ satisfying Hypothesis~\ref{HYP3}, \ref{HYP1}, \ref{HYP5} whose semi-Dirichlet form $(\mathcal{E},\mathcal{F})$ is quasi-regular. Denote the \emph{
 semigroup, co-semigroup} and \emph{resolvent, co-resolvent} (see the appendix) of $(\mathcal{E,F})$ by  $(T_t)_{t\geq0}$, $(\hat{T}_t)_{t\geq0}$ and $(G_\alpha)_{\alpha\geq 0}$, $(\hat{G}_\alpha)_{\alpha\geq0}$ respectively. All the other necessary notations and terminologies are given in the appendix.

To formulate Feynman-Kac formula, we need first establish the Revuz correspondence theory, which
was first done by Revuz in \cite{DR} and \cite{DR2} for \emph{positive continuous additive functionals}
(abbreviated as PCAF) under the duality assumption. Then the similar correspondence results relative to the general positive additive functionals (not necessarily to be continuous) and the multiplicative functionals under the duality assumption are formulated in  \cite{FG}, \cite{RGK2}, \cite{RGK3}, \cite{SMJ} and \cite{SMJ2}. The corresponding theory for PCAF's
in symmetric case was
developed by Fukushima in \cite{CM} and \cite{FU}. The main result is that each PCAF is in one-to-one
correspondence with a \emph{smooth} measure (also named by Revuz measure).
To discuss the killing transform by a discontinuous multiplicative functional, we have to
use the bivariate Revuz measures, which were first
introduced by Sharpe in \cite{SMJ} in dual case and
further discussed by the second author of this article in \cite{JGY} and \cite{JGY3}.

In this paper we shall consider the similar correspondence in the context of the semi-Dirichlet forms. The main difficulty is that the reference measure $m$ is not necessarily excessive for $X$.
However for any \emph{co-excessive function} $h$, the measure $h\cdot m$ is always excessive with respect to $X$
no matter $m$ is or is not excessive. The co-excessive functions are rich enough so that
it is possible for us to deal with the correspondence theory for semi-Dirichlet forms
similarly to the cases with duality assumption.
Actually the correspondence between the PCAFs and the  smooth measures was given in \cite{PJF}, \cite{MMS} and \cite{OY} in the context of the semi-Dirichlet forms.
Using the correspondence theory on PCAF's we shall treat general additive functionals and define
their bivariate Revuz measures on the semi-Dirichlet forms. 

The paper is organized as follows. In \S\ref{THWD} we shall focus on the transient case.
Although $X$ may not be transient we can consider the $1$-subprocess of $X$ which has the same properties as $X$ for the Revuz correspondence, see Proposition~\ref{PROPAM}.
Then under the transient assumption we can define the bivariate Revuz measures of the general additive functionals with respect to the reference measure in the context of the semi-Dirichlet forms.
As outlined in Theorem~\ref{EXIS} and Proposition~\ref{UNIQUE} such bivariate Revuz measure is unique. We shall also give some examples to characterize the bivariate Revuz measures of some typical additive functionals such as the Stieltjes logarithm of the multiplicative functional in \S\ref{EXAM}.

In \S\ref{FKF} we shall characterize the killing transform of the semi-Dirichlet forms. The killing transform of $X$ by a multiplicative functional $M$ is introduced in Appendix~\ref{DEFMF}.
We shall prove in Lemma~\ref{PQC} and \ref{HYLM4} that the resulting subprocess $X^M$ still satisfies Hypothesis~\ref{HYP3} and \ref{HYP5}. For the sector condition, i.e. Hypothesis~\ref{HYP1}, it will be more complicated.
In Theorem~\ref{FKT}, \ref{THMLX} and \ref{GCT} we shall give a sufficient condition to ensure that $X^M$ still satisfies the sector condition
and this sufficient condition is verified for the typical jump-type semi-Dirichlet forms (see Example~\ref{EMAP}) and all multidimensional diffusion processes with jumps outlined in \cite{UT} (see Example~\ref{EXA2}).
In particular we can use the bivariate Revuz measure of the stieltjes logarithm of $M$ to characterize the associated semi-Dirichlet form of $X^M$.
At last we shall extend the results in~\cite{JGY3} to the semi-Dirichlet forms in \S\ref{KAS}, which states that the killing transform in Markov processes is equivalent to the subordination in Dirichlet forms, see Theorem~\ref{EMSS}.

\section{Transience and weak duality}\label{THWD}




It is well known that even if $X$ is not transient, its subprocess $X^\delta$ killed by the MF $(e^{-\delta t})_{t\geq 0}$ is transient for arbitrary fixed constant
$\delta>0$. Clearly we have $$X^\delta=(\Omega, \mathcal{M},(\mathcal{M}_t)_{t\geq 0}, X_t, \theta_t, P^x_\delta)$$ where $P^x_\delta$ is defined by \eqref{POS} with $M=(\text{e}^{-\delta t})_{t\geq 0}$. 
In this section we shall illustrate that there is no difference between $X$ and $X^\delta$ in the context of the Revuz correspondence.

For the notation and terminology related to MF's and AF's, refer to \cite{JGY}. For example
we use $\text{MF}(X)$ to denote the set of exact decreasing MF's of $X$ and
\begin{align*}\text{MF}_+(X)&=\{M\in \text{MF}(X): M_0\equiv 1\};\\
\text{MF}_{++}(X)&=\{M\in \text{MF}(X): M_t>0\ \forall t<\zeta\}.
\end{align*}
The following lemma follows from \eqref{POS}.

\begin{lemma}\label{LFS}
	Let $\Gamma\in \mathcal{M}_t\cap (t<\zeta)$ for some $t\geq 0$. Then $P^x(\Gamma)=0$ if and only if $P_\delta^x(\Gamma)=0$ for any $x\in E$.
\end{lemma}

Then the following lemma follows directly from the perfect exact regularization outlined in (55.19) and (35.10) of \cite{SMJ2}.

\begin{lemma}\label{MMX}
For multiplicative functionals, it holds that \begin{align*}\text{MF}(X)&= \text{MF}(X^\delta)\\
\text{MF}_+(X)&=\text{MF}_+(X^\delta);\\
\text{MF}_{++}(X)&=\text{MF}_{++}(X^\delta).\end{align*}
Moreover $A\in \text{AF}(X,M)$ if and only if $A\in \text{AF}(X^\delta,M)$.
\end{lemma}

Let  $A\in\text{AF}(X,M)=\text{AF}(X^\delta, M)$ and $\xi \in \text{Exc}^\beta$ with some constant $\beta\geq 0$. Define the \emph{bivariate potential} $(\mathcal{U}_A^\alpha)_{\alpha\geq 0}$ of $A$ relative to $X$ by
\begin{equation}\label{BPA}
    \mathcal{U}_A^\alpha F(x):=E^x\int_0^\infty e^{-\alpha t}F(X_{t-},X_t)dA_t
\end{equation}
for any $F\in b(\mathcal{B}\times \mathcal{B})_+$.
Since $\xi$ is $\beta$-excessive the mapping
\begin{equation}\label{EQTMF}
    t\mapsto\frac{1}{t}e^{-\beta t}E^\xi \int_0^tF(X_{t-},X_t)dA_t
\end{equation}
is increasing as $t$ decreases. In particular
\[
	(F\ast A)_t:=\int_0^t F(X_{s-},X_s)dA_s,\quad t\geq 0\]
is an additive functional relative to the MF $M$, i.e. $F\ast A\in \text{AF}(M)$. Denote the limitation of \eqref{EQTMF} when $t\downarrow 0$ by $L_A(F)$. Then clearly
\[
    L_A(F)=\lim_{t\rightarrow 0}\frac{1}{t}e^{-\beta t}E^\xi \int_0^tF(X_{t-},X_t)dA_t=\lim_{t\rightarrow 0}\frac{1}{t}E^\xi \int_0^tF(X_{t-},X_t)dA_t
\]
and moreover
\begin{equation}\label{LAF}
    L_A(F)=\lim_{\alpha\rightarrow \infty}\alpha \xi \mathcal{U}_A^\alpha F.
     \end{equation}
Hence there exists the celebrated \emph{bivariate Revuz measure}, denoted by $\nu_A^\xi$, of $A$ on $E\times E$ with respect to $\xi$ such that
\begin{equation}\label{LAF2}
    L_A(F)=\int F(x,y)\nu_A^\xi(dxdy).
\end{equation}
Note that if $A$ is $\text{PCAF}$, $\nu_A^\xi$ is concentrated on the diagonal $d$  and
\[\nu_A^\xi(F)=\displaystyle \lim_{t\rightarrow 0}\frac{1}{t}e^{-\beta t}E^\xi \int_0^tF_D(X_t)dA_t\]
for any positive $F$ where $F_D(x):=F(x,x)$ for any $x\in E$. In other words,
\[	
	\nu_A^\xi(F)= \mu_A^\xi(F_D)
\]
 where $\mu_A^\xi$ is the classical Revuz measure of a PCAF $A$ with respect to $\xi$. Similarly let $\mathcal{U}^\alpha_{\delta,A}$ be the bivariate potential of $A$ with respect to $X^\delta$. Then for any $F\in b\mathcal{B}\times \mathcal{B}_+$ it follows that
\[\begin{aligned}
	\mathcal{U}^\alpha_{\delta,A}F(x)&=P^x_\delta \int_0^\infty \text{e}^{-\alpha t}F(X_{t-},X_t)dA_t   \\
		&=E^x\int_0^\infty [( \int_0^\infty \text{e}^{-\alpha t}F(X_{t-},X_t)dA_t )\circ k_s] d(-\text{e}^{-\delta s})  \\
		&=E^x\int_0^\infty ( \int_0^s \text{e}^{-\alpha t}F(X_{t-},X_t)dA_t ) d(-\text{e}^{-\delta s})  \\
		&=E^x\int_0^\infty \text{e}^{-\alpha t}F(X_{t-},X_t)dA_t \int_t^\infty d(-\text{e}^{-\delta s}) \\
		&=\mathcal{U}^{\alpha+\delta}_AF(x).
\end{aligned}\]
Therefore from \eqref{LAF}, \eqref{LAF2} we can deduce the following proposition.

\begin{proposition}\label{PROPAM}
	Assume $M\in \text{MF}(X)$ or $\text{MF}(X^\delta)$ and $A\in \text{AF}(X,M)$ or  $\text{AF}(X^\delta,M)$. Let $\xi\in \text{Exc}^\alpha(X)\subset\text{Exc}^{\alpha}(X^\delta)$ and $\nu_A^\xi,\nu_{\delta,A}^\xi$ the bivariate Revuz measures of $A$ relative to $X$ and $X^\delta$ respectively. Then $\nu_A^\xi=\nu_{\delta,A}^\xi$.
\end{proposition}

Without loss of generality we could always assume that the following transient assumption holds when discussing the Revuz measures or bivariate Revuz measures relative to $X$ and the ($\alpha$-)excessive measure $\xi$.

\begin{hypothesis}[Transience]\label{HYP2}
   There is a strictly positive function $g\in b\mathcal{B}$ such that $Ug$ is everywhere finite where $U$ is the potential kernal of $X$.
\end{hypothesis}

 Since $X$ is transient it follows from  \cite{BLNB} or Theorem~3.3.6 of  \cite{OY} that there exist a q.e. strictly positive q.c. coexcessive function $\hat{g}\in \mathcal{F}_\text{e}$ and an $m$-standard Markov process $\check{X}$ such that $X$ and $\check{X}$ are in weak duality relative to $\hat{g}\cdot m$. Clearly $\hat{m}:=\hat{g}\cdot m\in \text{Exc}$ and it is equivalent to $m$ since $\hat{g}$ is strictly positive. Thus a property holds $P^m$-a.s if and only if it holds $P^{\hat{m}}$-a.s. Moreover since every semipolar set is $m$-polar (equivalently, $\hat{m}$-polar) we have the following lemma.

    \begin{lemma}\label{DECOM}
         Any $M\in \text{MF}_+$ has a decomposition
        \begin{equation}\label{MFD}
            M_t=\prod_{0<s\leq t}(1-\Phi(X_{s-},X_s))\text{exp}\{-\int_0^t a(X_s)dA_s\}1_{[0,J_B)}(t)
        \end{equation}
        where $\Phi\in \mathcal{B}\times \mathcal{B},0\leq \Phi <1$, $\Phi$ vanishes on the diagonal $d$ of $E\times E$, $a\in \mathcal{B}_+$, $A$ is a continuous additive functional of $X$, $B$ is a Borel subset of $E\times E$ which is disjoint from $d$ and $S_M=J_B:=\inf\{t>0: (X_{t-},X_t)\in B\}$.
    \end{lemma}
    \begin{proof}
       Since $X$ is in weak duality to $\check{X}$ relative to $\hat{m}$ and every semipolar set is $\hat{m}$-polar it follows from \text{Theorem 2.2} of  \cite{JGY} (also see Theorem~7.1 of  \cite{SMJ}) that \eqref{MFD} holds $P^{\hat{m}}$-a.s. Hence it also holds in the sense of $P^m$-a.s.  
    \end{proof}

    \begin{corollary}\label{SLD}
        If $M\in \text{MF}_+$ has the decomposition \eqref{MFD}, then the Stieltjes logarithm of $M$, denoted by $[M]$, is
        \begin{equation}\label{BMD3}
            [M]_t=\sum_{s\leq t}\Phi(X_{s-},X_s)1_{\{s<S_M\}}+\int_0^t1_{\{s<S_M\}}a(X_s)dA_s.
        \end{equation}
    \end{corollary}

\section{Bivariate Revuz measure}\label{BRM2}

Throughout this section let $X$ be a right Markov process on $E\cup \{\Delta\}$ such that Hypothesis \ref{HYP3}, \ref{HYP1},  \ref{HYP5}  and \ref{HYP2} hold and its associated semi-Dirichlet form $(\mathcal{E},\mathcal{F})$ on $L^2(E,m)$ is quasi-regular.

\subsection{Existence and uniqueness}

Fix a multiplicative functional $M\in\text{MF}$. First  we have the following definition.

\begin{definition}
   An additive functional $A\in \text{AF}(M)$ is said to be \emph{integrable} if $\nu_A^{\hat{m}}(1)<\infty$ and \emph{$\sigma$-integrable} if we can write $E\times E=\cup_{i=1}^\infty F_i,\, F_i\in \mathcal{B}\times \mathcal{B}$ such that $\nu_A^{\hat{m}}(1_{F_i})<\infty$ for each $i$.
\end{definition}

The following theorem is our main result about the existence of the bivariate Revuz measure of $A\in \text{AF}(M)$ relative to $m$. Note that $F(x,y):=0$ if either $x=\Delta$ or $y=\Delta$.

\begin{theorem}\label{EXIS}
    Let $M\in \text{MF}$ and $A\in \text{AF}(M)$. Then there exists a unique positive measure $\nu_A$ on $E\times E$ charging no $m$-bipolar sets such that
    \begin{equation}\label{BRM}
    \begin{aligned}
        \int F(x,y)\tilde{\hat{h}}(x)\nu_A(dxdy)&=\lim_{t\downarrow 0}\frac{1}{t}E^{\hat{h}\cdot m}\int_0^tF(X_{s-},X_{s})dA_s\\&=\lim_{\alpha\uparrow \infty}\alpha(\hat{h},\mathcal{U}^\alpha_AF)_m
        \end{aligned}
    \end{equation}
    for any strictly positive $\gamma$-coexcessive function $\hat{h}$ with some constant $\gamma\geq 0$ and $F\in (\mathcal{B}\times \mathcal{B})_+$ where $\tilde{\hat{h}}$ is the q.c. $m$-version of $\hat{h}$ defined in \text{Remark \ref{IFW}}. In particular $A$ is $\sigma$-integrable if and only if $\nu_A$ is $\sigma$-finite. When $A$ is $\sigma$-integrable, \eqref{BRM} holds for any $\gamma$-coexcessive function $\hat{h}$ with some constant $\gamma\geq 0$ which is not necessarily strictly positive.
\end{theorem}

\begin{proof}
    Fix a $\gamma$-coexcessive function $\hat{h}$ with some constant $\gamma\geq 0$.  Without loss of generality we assume that $\hat{h}$ is quasi-continuous. Then clearly $\hat{h}\cdot m$ is a $\gamma$-excessive measure relative to $X$. It follows that (see II.1 of  \cite{DR} or  \cite{RGK3}) the mapping
    \[
    	t\mapsto\frac{1}{t}\text{e}^{-\gamma t}E^{\hat{h}\cdot m}\int_0^tF(X_{s-},X_{s})dA_s
    	\]
     is increasing as $t\downarrow 0$ and the mapping
     \[
     \alpha\mapsto\alpha(\hat{h},\mathcal{U}^{\alpha+\gamma}_AF)_m\]
      is increasing as $\alpha\rightarrow \infty$. Moreover their limitations are equal (may be infinite) and we can deduce that
    \[\begin{aligned}
        \lim_{t\downarrow 0}\frac{1}{t}&E^{\hat{h}\cdot m}\int_0^tF(X_{s-},X_{s})dA_s\\=&\lim_{t\downarrow 0}\frac{1}{t}\text{e}^{-\gamma t}E^{\hat{h}\cdot m}\int_0^tF(X_{s-},X_{s})dA_s \\
                =&\lim_{\alpha\rightarrow \infty}\alpha(\hat{h},\mathcal{U}^{\alpha+\gamma}_AF)_m   \\
                        =&\lim_{\alpha\rightarrow \infty}\alpha(\hat{h},\mathcal{U}^\alpha_AF)_m.
    \end{aligned}\]
    Hence the second equality in \eqref{BRM} holds.
To prove the first equality of \eqref{BRM} we first  assume that $\hat{h}$ is q.e. strictly positive. For any $F\in (\mathcal{B}\times \mathcal{B})_+$, define
    \begin{equation}\label{LF}
        L(F):=\lim_{\alpha\rightarrow \infty}\alpha(\hat{h},\mathcal{U}_A^{\alpha+\gamma}(F/\hat{h}))_m.
    \end{equation}
Note that $F/\hat{h}$ is q.e. positive. Since the value in the right side of  \eqref{LF} is increasing as $\alpha\uparrow \infty$, it follows from the monotone convergence theorem that
    \[
        L(\sum_{n=1}^\infty F^n)=\sum_{n=1}^\infty L(F^n)
    \]
    for any $F^n\in (\mathcal{B}\times \mathcal{B})_+,\;n\geq 1$. Obviously $L(0)=0$. Therefore there exists a positive measure denoted by $\nu_A$ on $E\times E$ such that
    \begin{equation}\label{EQLFI}
        L(F)=\int F(x,y)\nu_A(dxdy).
    \end{equation}
Replacing $F$ by $F(x,y)\hat{h}(x)$ in \eqref{EQLFI} we can deduce that
    \[
    \int F(x,y)\hat{h}(x)\nu_A(dxdy)=\lim_{\alpha\rightarrow \infty}\alpha(\hat{h},\mathcal{U}_A^{\alpha+\gamma}(F))_m.
    \]
 We claim that $\nu_A$ is independent of the choice of $\hat{h}$. In fact let $\hat{h}_1,\hat{h}_2$ be two $\gamma$-coexcessive q.e. strictly positive functions in $\mathcal{F}_\text{e}$ and $\nu_A^1,\nu_A^2$ the corresponding measures satisfying \eqref{BRM}. Note that $\hat{h}_i$ should be $\gamma_i$-coexcessive for $i=1,2$ with two constants $\gamma_1,\gamma_2$ such that $\gamma_1\leq \gamma_2$. But it follows that $\hat{h}_1$ is also $\gamma_2$-coexcessive. Hence $\hat{h}_1,\hat{h}_2$ are both $\gamma$-coexcessive for $\gamma=\gamma_2$. Take $F\in b(\mathcal{B}\times \mathcal{B})_+$ and without loss of generality we can assume that
 \[
 	\int F(x,y)\hat{h}_2(x)\nu^2_A(dxdy)<\infty.
 	\]
 	Otherwise \eqref{FAH} always holds.
  Then for any $t>0$ it follows that
 \[
 	E^{\hat{h}_2\cdot m}\int_0^t F(X_{s-},X_s)dA_s\leq t\int F(x,y)\hat{h}_2(x)\nu^2_A(dxdy)<\infty
\] 	and hence
    \[
        f_t(x):=E^x\int_0^t F(X_{s-},X_s)dA_s<\infty\quad m\text{-a.e. }x.
    \]
On the other hand,
\[\begin{aligned}
E^{\hat{h}_1\cdot m}&\int_0^t \text{e}^{-\gamma t}(F\cdot \hat{h}_2/\hat{h}_1)(X_{s-},X_s)dA_s \\=&E^{\hat{h}_1\cdot m}\int_0^t \text{e}^{-\gamma t}(\hat{h}_2/\hat{h}_1)(X_{s-})d(F\ast A)_s \\
                    =& E^{\hat{h}_1\cdot m}\lim_{n\rightarrow \infty}\sum_{k=1}^{n} \text{e}^{-\gamma kd_n}(\hat{h}_2/\hat{h}_1)(X_{(k-1)d_n})((F\ast A)_{kd_n}-(F\ast A)_{(k-1)d_n})
\end{aligned}\]
where $d_n=t/n$.
  It follows from Fatou Lemma and the Markov property of $X$ that
    \[\begin{aligned}
         E^{\hat{h}_1\cdot m}&\int_0^t \text{e}^{-\gamma t}(F\cdot \hat{h}_2/\hat{h}_1)(X_{s-},X_s)dA_s\\
                    \leq& \liminf_{n\rightarrow \infty}\sum_{k=0}^{n-1} E^{\hat{h}_1\cdot m} [\text{e}^{-\gamma kd_n}\frac{\hat{h}_2}{\hat{h}_1}(X_{kd_n})((F\ast A)_{(k+1)d_n}-(F\ast A)_{kd_n})]  \\
                    =&\liminf_{n\rightarrow \infty}\sum_{k=0}^{n-1} E^{\hat{h}_1\cdot m}[\text{e}^{-\gamma kd_n}\frac{\hat{h}_2}{\hat{h}_1}(X_{kd_n})(M_{kd_n} (F\ast A)_{d_n}\circ \theta_{kd_n})] \\
                    \leq &\liminf_{n\rightarrow \infty}\sum_{k=0}^{n-1} E^{\hat{h}_1\cdot m}[ \text{e}^{-\gamma kd_n}(\hat{h}_2/\hat{h}_1)(X_{kd_n})((F\ast A)_{d_n}\circ \theta_{kd_n})]   \\
                    = &\liminf_{n\rightarrow \infty}\sum_{k=0}^{n-1} E^{\hat{h}_1\cdot m}[ \text{e}^{-\gamma kd_n}(\hat{h}_2/\hat{h}_1)(X_{kd_n})E^{X_{kd_n}}(F\ast A)_{d_n}].
    \end{aligned}\]
    Define
    \[D_j:=\{f_{d_n}<j\}\cap \{\frac{1}{j}\leq \hat{h}_1\leq j\}\cap \{\frac{1}{j}\leq \hat{h}_2\leq j\}.\]
    Clearly $D_j\uparrow E\;m$-a.e. as $j\rightarrow \infty$ and $(f_{d_n}\cdot \hat{h}_2/\hat{h}_1)\cdot 1_{D_j}\in bL^2(E,m)$. Since $\hat{h}_1$ is $\gamma$-coexcessive it follows that
    \[\begin{aligned}
        E^{\hat{h}_1\cdot m}&[ \text{e}^{-\gamma kd_n}(\hat{h}_2/\hat{h}_1)(X_{kd_n})E^{X_{kd_n}}(F\ast A)_{d_n}]\\=&E^{\hat{h}_1\cdot m}[ \text{e}^{-\gamma kd_n}(f_{d_n}\hat{h}_2/\hat{h}_1)(X_{kd_n})] \\
        =& \lim_{j\rightarrow \infty}E^{\hat{h}_1\cdot m}[ \text{e}^{-\gamma kd_n}(f_{d_n}\hat{h}_2/\hat{h}_1\cdot 1_{D_j})(X_{kd_n})] \\=& \lim_{j\rightarrow \infty}(\text{e}^{-\gamma kd_n}\hat{h}_1, P_{kd_n}(f_{d_n}\hat{h}_2/\hat{h}_1\cdot 1_{D_j}))_m \\
        =& \lim_{j\rightarrow \infty}(\text{e}^{-\gamma kd_n}\hat{h}_1, T_{kd_n}(f_{d_n}\hat{h}_2/\hat{h}_1\cdot 1_{D_j}))_m  \\
        =& \lim_{j\rightarrow \infty}(\text{e}^{-\gamma kd_n}\hat{T}_{kd_n}\hat{h}_1, f_{d_n}\hat{h}_2/\hat{h}_1\cdot 1_{D_j})_m \\
        \leq & (\hat{h}_1,  f_{d_n}\hat{h}_2/\hat{h}_1)_m \\
        =&E^{\hat{h}_2\cdot m}(F\ast A)_{d_n}.
    \end{aligned}\]
    Thus we have
    \[\begin{aligned}
        E^{\hat{h}_1\cdot m}\int_0^t \text{e}^{-\gamma t}(F\cdot \hat{h}_2/\hat{h}_1)(X_{s-},X_s)dA_s
            &\leq \liminf_{n\rightarrow \infty}nE^{\hat{h}_2\cdot m}(F\ast A)_{d_n} \\
            &=t\liminf_{n\rightarrow \infty}\frac{1}{d_n}E^{\hat{h}_2\cdot m}(F\ast A)_{d_n}.
    \end{aligned}
    \]
    In other words,
    \[
        \frac{1}{t}E^{\hat{h}_1\cdot m}\int_0^t \text{e}^{-\gamma t}(F\cdot \hat{h}_2/\hat{h}_1)(X_{s-},X_s)dA_s\leq \lim_{s\rightarrow 0}\frac{1}{s}E^{\hat{h}_2\cdot m}(F\ast A)_s.
    \]
    Let $t\downarrow 0$ and we can deduce that
    \begin{equation}\label{FAH}
    \int F(x,y)\hat{h}_2(x)\nu_A^1(dxdy)\leq \int F(x,y)\hat{h}_2(x)\nu_A^2(dxdy).
     \end{equation}
     Similarly we conclude that
     \[
     	\int	F(x,y)\hat{h}_2(x)\nu_A^2(dxdy)\leq \int F(x,y)\hat{h}_2(x)\nu_A^1(dxdy).
     	\] Since $F$ is arbitrary,  we have $\hat{h}_2(x)\nu_A^1(dxdy)=\hat{h}_2(x)\nu_A^2(dxdy)$ and it follows  that  \[\nu_A^1=\nu_A^2.\] In particular the measure $\hat{g}(x)\nu_A(dxdy)=\nu_A^{\hat{m}}$ charges no $m$-bipolar sets. Then $\nu_A$ also charges no $m$-bipolar sets because $\hat{g}$ is strictly positive. The uniqueness of $\nu_A$ which satisfies  \eqref{BRM} is apparent.

     Note that a positive measure is $\sigma$-finite if and only if there  exists a strictly positive and integrable function relative to this measure. Thus if $A$ is $\sigma$-integrable there exists a strictly positive function $F$ such that
     \[
     	\nu_A^{\hat{m}}(F)<\infty.
     \]
On the other hand since
\[
	\int\hat{g}(x)F(x,y)\nu_A(dxdy)=\nu_A^{\hat{m}}(F)
\]
and $\hat{g}$ is also strictly positive we can deduce that $\nu_A$ is $\sigma$-finite. On the contrary we can similarly prove that if $\nu_A$ is $\sigma$-finite then $A$ is $\sigma$-integrable.

    Finally if $A$ is $\sigma$-integrable we assert that \eqref{BRM} holds for any $\gamma$-coexcessive function $\hat{h}$ which is not necessarily strictly positive. To this end define $\hat{h}_\epsilon=\hat{h}+\epsilon \hat{g}$ and clearly $\hat{h}_\epsilon$ is $\gamma$-coexcessive and strictly positive. Choose some function $F\in (\mathcal{B}\times \mathcal{B})_+$ such that $\nu_A^{\hat{m}}(F)=\int F(x,y)\hat{g}(x)\nu_A(dxdy)<\infty$. Since \eqref{BRM} holds for $\hat{h}_\epsilon$ it follows that
    \[\begin{aligned}
       \int F (x,&y)(\hat{h}(x)+\epsilon \hat{g}(x))\nu_A(dxdy)\\=&\lim_{\alpha\rightarrow \infty}\alpha(\hat{h}+\epsilon \hat{g},\mathcal{U}^\alpha_AF)_m \\=&\lim_{\alpha\rightarrow \infty}\alpha(\hat{h},\mathcal{U}^\alpha_AF)_m+\epsilon \lim_{\alpha\rightarrow \infty}\alpha(\hat{g},\mathcal{U}^\alpha_AF)_m.
    \end{aligned}\]
    Let $\epsilon\downarrow 0$ we can deduce that
    \[
    	\int F(x,y)\hat{h}(x)\nu_A(dxdy)=\lim_{\alpha\rightarrow \infty}\alpha(\hat{h},\mathcal{U}^\alpha_AF)_m.
    	\] That completes the proof. 
\end{proof}

\begin{definition}
    Let $M\in \text{MF}$ and $A\in \text{AF}(M)$. A positive measure $\nu_A$ on $E\times E$ is called the \emph{bivariate Revuz measure} of $A$ if \eqref{BRM} holds for any strictly positive $\gamma$-coexcessive function $\hat{h}$, $\gamma\geq 0$ and $F\in (\mathcal{B}\times \mathcal{B})_+$.
\end{definition}

We always  denote the bivariate Revuz measure of $A$ by $\nu_A$. When $A$ is integrable, i.e. $\hat{g}\cdot \nu_A(1)=\nu_A^{\hat{m}}(1)<\infty$, we do not have  $\nu_A(1)<\infty$ whereas $\nu_A$ is always $\sigma$-finite by \text{Theorem \ref{EXIS}}. Moreover we have the following useful corollary.

\begin{corollary}
    Let  $A\in\text{AF}(M)$ be $\sigma$-integrable and $F\in(\mathcal{B}\times \mathcal{B})_+$. Then the additive functional
    \[
    	(F\ast A)_t=\int_0^t F(X_{s-},X_s)dA_s,\quad t\geq 0
    	\]
  is $\sigma$-integrable and $\nu_{F\ast A}=F\cdot \nu_A$.
\end{corollary}
\begin{proof}
  Since $A$ is $\sigma$-integrable we can write $E\times E=\cup_{i=1}^\infty H_i$ where $H_i\in \mathcal{B}\times \mathcal{B}$ such that $\nu_A^{\hat{m}}(1_{H_i})<\infty$ for each $i$. Let $\Gamma_i:=H_i\cap (F\leq i)$ for each $i$. Then $E\times E=\cup_{i=1}^\infty \Gamma_i$ and we have
    \[\begin{aligned}
        \nu_{F\ast A}^{\hat{m}}(1_{\Gamma_i})&=\lim_{t\rightarrow 0}\frac{1}{t}E^{\hat{m}}\int_0^t 1_{\Gamma_i}\cdot F(X_{s-},X_s)dA_s \\
            &\leq i\lim_{t\rightarrow 0}\frac{1}{t}E^{\hat{m}}\int_0^t 1_{H_i}(X_{s-},X_s)dA_s \\
                &=i\nu_A^{\hat{m}}(1_{H_i})\\
                &<\infty.
    \end{aligned}
    \]
 Hence $F\ast A$ is $\sigma$-integrable. The second assertion is apparent.  
\end{proof}

We can also extend Theorem~A.8 of  \cite{MMS} to the bivariate Revuz measures. Note that in the following proposition $M\equiv 1$.

\begin{proposition}
    Let $A\in \text{AF}$ be $\sigma$-integrable and $\nu_A$ its bivariate Revuz measure. Then
    \begin{equation}\label{BRF}
        (h,\mathcal{U}_A^\alpha F)_m=\int\widetilde{\hat{G}_\alpha h}(x) F(x,y)\nu_A(dxdy)
    \end{equation}
    for any $h\in L^2(E,m)\cap\mathcal{B}_+,F\in (\mathcal{B}\times \mathcal{B})_+$ and  $\alpha\geq 0$. The following formula also holds
    \begin{equation}
        E^{h\cdot m}\int_0^tF(X_{s-},X_s)dA_s= \int_0^t\langle F\cdot \nu_A, \widetilde{\hat{T}_sh}\rangle ds
    \end{equation}
    where $\langle F\cdot \nu_A, \widetilde{\hat{T}_sh}\rangle=\int \widetilde{\hat{T}_sh}(x)F(x,y)\nu_A(dxdy)$.
\end{proposition}
\begin{proof}
    We only need to prove \eqref{BRF}. In fact since $\hat{G}_\alpha g$ is $\alpha$-coexcessive it follows from \eqref{BRM} that
    \[\begin{aligned}
        \int\widetilde{\hat{G}_\alpha h}(x) F(x,y)\nu_A(dxdy)&=\lim_{n\rightarrow \infty}n(\hat{G}_\alpha h, \mathcal{U}_A^{\alpha +n}F)_m\\&=\lim_{n\rightarrow \infty}n(h, G_\alpha \mathcal{U}_A^{\alpha +n}F)_m\\&=\lim_{n\rightarrow \infty}(h,\mathcal{U}_A^\alpha F-\mathcal{U}_A^{n+\alpha}F)_m\\&=(h,\mathcal{U}_A^\alpha F)_m.
    \end{aligned}\]
 The third equality is because of the formula (see Proposition 3.4 of  \cite{SMJ})
 \[\mathcal{U}_A^\alpha F-\mathcal{U}_A^{n+\alpha}F=nG_\alpha \mathcal{U}_A^{\alpha +n}F.\] That completes the proof.  
\end{proof}

The uniqueness of the correspondence of the additive functionals and bivariate Revuz measures is as follows.

\begin{proposition}[Uniqueness]\label{UNIQUE}
    Let $A^1,A^2\in \text{AF}$ be $\sigma$-integrable. Then $A^1$ and $A^2$ are $m$-equivalent if and only if their bivariate Revuz measures are equal, i.e. $\nu_{A^1}=\nu_{A^2}$.
\end{proposition}
\begin{proof}
    Note that $\nu_{A^1}=\nu_{A^2}$ if and only if $\nu^{\hat{m}}_{A^1}=\hat{g}\cdot\nu_{A^1}=\hat{g}\cdot\nu_{A^2} =\nu_{A^2}^{\hat{m}}$. The uniqueness is obvious by \text{Proposition 6.2} of  \cite{JGY}.  
\end{proof}

\begin{remark}
Note that from Theorem~\ref{EXIS} we can deduce that the function $\hat{g}$ in the definition of the $\sigma$-integrable additive functionals can be replaced by any other strictly positive $\alpha$-coexcessive function for any $\alpha\geq 0$.

Since $A$ does not charge $[S_M, \infty)$ it follows that $\nu_A$ is supported on $E_M\times E_M$. If $M\in \text{MF}_+$ then $E_M=E$ and
\[
	S_M=J_B:=\inf\{t>0: (X_{t-},X_t)\in B\}
	\]
where $B$ is a Borel subset of $E\times E$ and disjoint from the diagonal $d$ (see \text{Lemma \ref{DECOM}}). Hence $\nu_A$ is supported on $B^c$. Generally any $A\in \text{AF}(M)$ can be decomposed by
    \[
        A=A^c+A^n+A^q
    \]
    where $A^c\in \text{PCAF}(M)$, $A^n$ is a pure jump natural AF of $(X,M)$ and $A^q$ is a pure-jump AF of $(X,M)$ which is quasi-left-continuous (q.l.c.) in the sense that every discontinuity of the mapping $t\mapsto A_t^q$ is also a discontinuity of $t\mapsto X_t$. Note that under the sector condition every natural AF is continuous a.s. and hence we can write $A=\tilde{A}^c+A^q$ where $\tilde{A}^c=A^c+A^n$ is continuous. In particular the continuous part $\tilde{A}^c$ of $A$ is $\sigma$-integrable and its bivariate Revuz measure $\nu_{\tilde{A}^c}$ is supported on the diagonal $d$. On the other hand under some appropriate conditions (see \text{Theorem 5.1} of  \cite{SMJ}) the pure-jump part $A^q$ of $A$ is equivalent to an AF
    \[
    	C_t=\sum_{s\geq t}\Upsilon(X_{s-},X_s)M_s,\quad t\geq 0
    	\] where $\Upsilon\in (\mathcal{B}\times \mathcal{B})_+$ is a function carried by $E_M\times E_M$, finite everywhere and vanishes on $d$. In particular under the same conditions $A^q$ is $\sigma$-integrable and thus $A$ is also $\sigma$-integrable.
\end{remark}

\begin{proposition}[Proposition 5.6,  \cite{SMJ}]
    Assume that the resolvent $U(x,dy)$ of $X$ is absolutely continuous with respect to $m(dy)$ for $m$-a.e. $x$. If $A\in \text{AF}(M)$ and $A_t=A_{S_M-}$ for any $ t\geq S_M$, then $A$ is $\sigma$-integrable.
\end{proposition}

\subsection{Left and right Revuz measures}

Let $A\in \text{AF}(M)$ and $\hat{h}$ a $\gamma$-coexcessive function for some constant $\gamma\geq 0$. We can define the \emph{left Revuz measure} $\lambda_A^{\hat{h}\cdot m}$ and \emph{right Revuz measure} $\rho_A^{\hat{h}\cdot m}$ of $A$ relative to $\hat{h}\cdot m$ by
\begin{equation}\label{LRM}
    \lambda_A^{\hat{h}\cdot m}(f):=\uparrow \lim_{t\downarrow 0}\frac{1}{t}E^{\hat{h}\cdot m}\int_0^tf(X_{s-})dA_s
\end{equation}
and
\begin{equation}
    \rho_A^{\hat{h}\cdot m}(f):=\uparrow \lim_{t\downarrow 0}\frac{1}{t}E^{\hat{h}\cdot m}\int_0^tf(X_s)dA_s
\end{equation}
for any $f\in \mathcal{B}_+$. The left Revuz measure is also called the \emph{Revuz measure} in abbreviation. Note that we need to assume that $X_{\zeta-}$ exists in \eqref{LRM} if $A$ charges $\zeta$.
Similarly to \text{Theorem \ref{EXIS}} we can deduce that if $\lambda_A^{\hat{h}\cdot m}$ is $\sigma$-finite then there exists a $\sigma$-finite measure $\lambda_A$ on $E$ charging no $m$-polar sets such that
\begin{equation}\label{RM}
    \lambda_A^{\hat{h}\cdot m}(f)=\lambda_A(\hat{h}\cdot f)
\end{equation}
for any $\gamma$-coexcessive q.c. function $\hat{h}$, $\gamma\geq 0$ and $f\in \mathcal{B}_+$. We call $\lambda_A$  the \emph{Revuz Measure} of $A$ \emph{relative to} $m$. We also define the \emph{left} and \emph{right marginal measures} $\nu_A^1$ and $\nu_A^2$ of $\nu_A$ on $E$ by
\[	
	\nu_A^1(f):=\nu_A(f\otimes 1),\quad \nu_A^2(f):=\nu_A(1\otimes f)
	\]  for any $f\in\mathcal{B}_+$. Here $(f\otimes 1)(x,y):=f(x),(1\otimes f)(x,y):=f(y)$ for any $x,y\in E$. Clearly $\nu_A^1$ and $\nu_A^2$ charge no $m$-polar sets and it follows from  \eqref{BRM} and \eqref{RM} that
\[
\lambda_A(\hat{h}\cdot f)=\lambda_A^{\hat{h}\cdot m}(f)=\nu_A^{\hat{m}}(f\otimes 1)=\nu_A(f\hat{h}\otimes 1)=\nu_A^1(\hat{h}\cdot f).
\]
 Hence we have the following proposition.

\begin{proposition}
    $\lambda_A=\nu_A^1$.
\end{proposition}

However we cannot obtain similar results about the right Revuz measures (i.e. $\rho_A=\nu_A^2$). To see this let $\hat{h}$ and $f$ be above and assume that $A$ does not charge $\zeta$, i.e. $A_\zeta-A_{\zeta-}=0$. Then we have
\[
\begin{aligned}
    \rho_A^{\hat{h}\cdot m}(f)&=\uparrow \lim_{t\downarrow 0}\frac{1}{t}E^{\hat{h}\cdot m}\int_0^tf(X_s)dA_s\\&=\nu_A^{\hat{h}\cdot m}(1\otimes f)\\&=\nu_A(\hat{h}\otimes f)\\&\neq \nu_A^2(\hat{h}\cdot f).
\end{aligned}\]
However on the other hand if $A$ is continuous then apparently
\begin{equation}\label{LRRM}
    \lambda_A=\nu_A^1=\nu_A^2.
\end{equation}
In particular the Revuz measure and right Revuz measure of $A$  are the same. Moreover if $A$ is a PCAF of $X$  then the measure in \eqref{LRRM} is exactly the smooth measure corresponding to $A$ introduced in Appendix~\ref{PCAF}.

\subsection{Examples}\label{EXAM}
In this section we assume that $M\in \text{MF}_+$ has the decomposition \eqref{MFD} in \text{Lemma \ref{DECOM}}, i.e.
\[
     M_t=\prod_{0<s\leq t}(1-\Phi(X_{s-},X_s))\text{exp}\{-\int_0^t a(X_s)dA_s\}1_{[0,J_B)}(t)
\]
with some functions $\Phi,a$, PCAF $A$ and a subset $B$ of $E\times E$. We shall compute the bivariate Revuz measures of some typical AFs and the primary tool is \emph{L\'{e}vy system}.  L\'{e}vy system is used to  characterize the discontinuous part of the Markov process. It is a pair $(N,H)$ for $X$ where $N$ is a kernel on $(E,\mathcal{B})$ such that $N(x,\{x\})=0$ for any $x\in E$ and $H$ is a PCAF of $X$ such that the 1-potential of $H$ is bounded and  for any $F\in (\mathcal{B}\times \mathcal{B})_+$, any predictable process $Y$ and $x\in E$,
\begin{equation}\label{LS}
    E^x\sum_{0<s\leq t}Y_sF(X_{s-},X_s)=E^x\int_0^tY_sdH_s \int F(X_s,y)N(X_s,dy).
\end{equation}
Let $\mu_H$ be the corresponding smooth measure, i.e. Revuz measure, of the PCAF $H$ and define
\begin{equation}\label{CM}
    \nu(dxdy):=N(x,dy)\mu_H(dx).
\end{equation}
The measure $\nu$ is called the \emph{canonical measure} of $X$. Clearly $\nu$ is a $\sigma$-finite measure supported on $E\times E \setminus d$ and charges no $m$-bipolar sets. It follows from \eqref{LS} and \eqref{RFC} that for any $F\in (\mathcal{B}\times \mathcal{B})_+$ and $\gamma$-coexcessive function $\hat{h}$ with some constant $\gamma\geq 0$,
\[
    \int F(x,y)\hat{h}(x)\nu(dxdy)=\lim_{t\downarrow 0}\frac{1}{t}E^{\hat{h}\cdot m}\sum_{s\leq t}F(X_{s-},X_{s}).
\]
Let $X^M=(X,M)$ be the subprocess of $X$ killed by $M$. By \text{Theorem~3.4} of  \cite{JGY} the L\'{e}vy system of $X^M$ is $(N_0,H)$ where $N_0$ is given by
\[
N_0(x,dy)=(1_{B^c}-1_{B^c}\cdot \Phi)(x,y)N(x,dy).
\]
Moreover the canonical measure of $X^M$ is
\begin{equation}\label{CMSP}
	\nu^M(dxdy)=(1_{B^c}-1_{B^c}\cdot \Phi)(x,y)\nu(dxdy).
\end{equation}
Note that $\bar{M}:=1-M$ is an AF of $(X,M)$ and  Stieltjes logarithm $[M]$ is an AF of $(X,S_M)$. Similarly to  \cite{JGY} and the proof  of Theorem~\ref{EXIS} we can deduce the following characterizations for $\nu_{\bar{M}}$ and $\nu_{[M]}$.

\begin{proposition}\label{SLMD3}
    Assume that $M\in \text{MF}_+$ has the decomposition \eqref{MFD}. Then its Stieltjes logarithm $[M]$ is $\sigma$-integrable and the bivariate Revuz measure of $[M]$ is
    \begin{equation}\label{SLMD2}
        \nu_{[M]}(dxdy)=1_{B^c}(x,y)\cdot \Phi(x,y)\cdot \nu(dx dy)+\delta_y(dx)a(y)\mu_A(dy)
    \end{equation}
    where $\nu$ is the canonical measure of $X$ defined by \eqref{CM}, $\delta_y$ is the point mass of $\{y\}$ and $\mu_A$ is the smooth measure associated with PCAF $A$. In particular if $M\in \text{MF}_{++}$ then
    \begin{equation}\label{SLMD4}
        \nu_{[M]}(dxdy)=\Phi(x,y)\cdot \nu(dx dy)+\delta_y(dx)a(y)\mu_A(dy).
    \end{equation}
\end{proposition}

\begin{proposition}\label{USC}
	Under the same conditions as in Proposition~\ref{SLMD3}   the additive functional $\bar{M}$ is $\sigma$-integrable and its  bivariate Revu measure is
\begin{equation}\label{BMD}
	\nu_{\bar{M}}(dxdy)=(1_B+1_{B^c}\cdot \Phi)(x,y)\nu(dxdy)+\delta_y(dx)a(y)\mu_A(dy).
\end{equation}
In particular if $M_t=1_{\{t<S_M\}}$, i.e. $\Phi=0,a=0$, then the bivariate Revuz measure of $(1_{\{t<S_M\}})_{t\geq 0}$ is
\begin{equation}\label{BMD2}
	\nu_{S_M}(dxdy)=1_B(x,y)\nu(dxdy).
\end{equation}
\end{proposition}

Since $B$  is disjoint to the diagonal $d$ of $E\times E$ it follows from  \eqref{CMSP}, \eqref{SLMD2}, \eqref{BMD}  and \eqref{BMD2} that
\begin{equation}\label{EQVMM}
	\nu_{\bar{M}}=\nu_{[M]}+\nu_{S_M}
\end{equation}
and
\[
	\nu^M+1_{E\times E\setminus d}\cdot \nu_{\bar{M}}=\nu.
	\]
	If in addition $M\in \text{MF}_{++}$ then
	\begin{equation}\label{MME}
		\nu_{\bar{M}}=\nu_{[M]}.
		\end{equation}
In fact  \eqref{EQVMM} still holds even if $M$ is only in MF.

\begin{proposition}\label{VMS}
    If $M\in \text{MF}$ then $\nu_{\bar{M}}=\nu_{[M]}+\nu_{S_M}$.
\end{proposition}
\begin{proof}
    It follows from Theorem~4.17(ii) of  \cite{JGY}  that $\nu_{\bar{M}}^{\hat{g}\cdot m}=\nu_{[M]}^{\hat{g}\cdot m}+\nu_{S_M}^{\hat{g}\cdot m}$. Thus by \text{Theorem \ref{EXIS}} we may conclude that $\nu_{\bar{M}}=\nu_{[M]}+\nu_{S_M}$.  
\end{proof}

\section{Feynman-Kac formula}\label{FKF}

Throughout this section let $X$ be a right Markov process satisfying Hypothesis~\ref{HYP3}, \ref{HYP1} and \text{\ref{HYP5}} whose associated  semi-Dirichlet form $(\mathcal{E},\mathcal{F})$ on $L^2(E,m)$ is quasi-regular. When necessarily, we may always take its quasi-continuous version for any function in $\mathcal{F}$. Further fix $M\in \text{MF}$. Then as outlined in Appendix~\ref{DEFMF} we use $X^M$ or $(X,M)$ to denote the subprocess of $X$ killed by $M$. Clearly $X^M$ is a right Markov process. In fact if $X$ satisfies \text{Hypothesis \ref{HYP3}} then the subprocess $X^M$ also satisfies \text{Hypothesis \ref{HYP3}}.

\begin{lemma}\label{PQC}
    Let $X$ be above, $M\in \text{MF}$, $E_M$ the set of all the permanent points of $M$ and $m^*:=m|_{E_M}$. Denote the semigroup of the subprocess $(X,M)$ by $(Q_t)_{t\geq 0}$. Then $(Q_t)_{t\geq 0}$ can be extended to a strongly continuous contraction semigroup on $L^2(E_M,m^*):=\{u\in L^2(E,m):u|_{E_M^c}=0\}$.
\end{lemma}
\begin{proof}
    Note that $L^2(E_M,m^*)=pL^2(E_M,m^*)-pL^2(E_M,m^*)\subset L^2(E,m)$. For any $f\in pL^2(E_M,m^*)\subset L^2(E,m)$ and $x\in E_M$ we have
    \[
        Q_tf(x)=E^x(f(X_t)M_t)\leq E^x(f(X_t))=P_tf(x)\in L^2(E,m).
    \]
Clearly $Q_tf(x)=0$ for any $x\in E_M^c$. Hence $Q_tf\in L^2(E_M,m^*)$. Moreover the semigroup property of $(Q_t)_{t\geq 0}$, i.e. $Q_tQ_s=Q_{t+s}$ for any $t,s\geq 0$, is apparent. For any $f\in L^2(E_M,m^*)$ and $x\in E_M$ it follows that
    \[
    	|Q_tf(x)|= |E^x(f(X_t)M_t)|\leq E^x(|f|(X_t))=P_t|f|(x)
    	\]
and hence
    \[
        \int |Q_tf(x)|^2 m^*(dx)\leq \int (P_t|f|(x))^2 m^*(dx)\leq ||f||^2_{L^2(E_M,m^*)}.
    \]
    This is the contraction property of $(Q_t)_{t\geq 0}$. At last  we claim that $(Q_t)_{t\geq 0}$ is strongly continuous on $L^2(E_M,m^*)$.  Since $(Q_t)$ is contractive on $L^2(E,m)$ we only need to prove the strongly continuous property of $(Q_t)_{t\geq 0}$ on a dense subset of $L_+^2(E,m)$ with respect to $L^2$-norm. Set
    \[
    	\mathcal{C}:=\{U^1f: f\in bL^2_+(E,m) \}\]
  where $U^1$ is the $1$-potential of $X$. Clearly $\mathcal{C}$ is dense in $L_+^2(E,m)$ with respect to $L^2$-norm. Choose an $\mathcal{E}$-nest $\{F_n\}$ such that $m(F_n)<\infty$ for any $n\geq 1$, which may be constructed by a q.e. strictly positive and q.c function $g\in\mathcal{F}$, say $F_n:=\{g\geq \frac{1}{n}\}$.  For any $u=U^1f\in \mathcal{C}\subset \mathcal{F}$ define
  \[
  		u_n:=u-P^1_{F_n^c}u\in \mathcal{F}
  		\]
 where $P^1_{F^c_n}$ is the balayage operator and it follows that
    \[
        u_n(x)=E^x\int_0^{T_{F_n^c}}e^{-t}f(X_t)dt,\quad x\in E,
    \]
where $T_{F_n^c}$ is the hitting time of $F_n^c$. Clearly $u_n$ is quasi-continuous, $u_n|_{F_n^c}=0$, $u_n\uparrow u$ $m$-a.e. and hence $u_n\rightarrow u$ in $L^2(E,m)$. But $u$ is bounded and $m(F_n)<\infty$. Thus we can deduce that $u_n\in bL^1(E,m)$ and
 \[
 	|Q_tu_n(x)\cdot u_n(x)|\leq ||u||_\infty\cdot u_n\in L^1(E,m).\]
Since $u_n(X_\cdot)$ is right continuous it follows that for any $x\in E_M$,
    \[
        \lim_{t\downarrow 0} Q_t u_n(x)=\lim_{t\downarrow 0}E^x(u_n(X_t)M_t)=E^x\lim_{t\downarrow 0}(u_n(X_t)M_t)=u_n(x).
    \]
By the dominated convergence theorem and the contraction property of the semigroup $(Q_t)_{t\geq 0}$ we have
    \[\begin{aligned}
      \lim_{t\downarrow 0}\int  & (Q_tu_n(x)-u_n(x))^2m^*(dx)\\=&\lim_{t \downarrow 0}\int [(Q_tu_n(x))^2+(u_n(x))^2-2Q_tu_n(x)\cdot u_n(x)]m^*(dx) \\
                        \leq&\lim_{t \downarrow 0}\int [2(u_n(x))^2-2Q_tu_n(x)\cdot u_n(x)]m^*(dx)  \\
                        =&\int \lim_{t \downarrow 0}[2(u_n(x))^2-2Q_tu_n(x)\cdot u_n(x)]m^*(dx) \\
                        =&0.
    \end{aligned}\]
Since $u_n\rightarrow u$ in $L^2(E,m)$ it follows from the contraction property of $(Q_t)_{t\geq 0}$ again that $Q_t u \rightarrow u$ in $L^2(E,m)$ as $n\rightarrow \infty$.  
\end{proof}

Similarly we can prove that if $X$ satisfies \text{Hypothesis \ref{HYP5}} so does $X^M$.

\begin{lemma}\label{HYLM4}
 If $X$ satisfies \text{Hypothesis \ref{HYP5}}, i.e. \eqref{MT} holds, and $M\in \text{MF}$ then the subprocess $X^M$ also satisfies \text{Hypothesis \ref{HYP5}}.
\end{lemma}
\begin{proof}
    Note that the killing transform by $M$ can be completed in two steps: killing $X$ by a hitting time $T_{E_M^c}$ firstly and killing then by a multiplicative functional in $\text{MF}_+$. The first step has been discussed in \text{Theorem 5.10} of  \cite{PJF}. Hence we only need to prove it for $M\in \text{MF}_+$. To this end let $\{K_n:n\geq 1\}$ be a sequence of subsets of $E$ satisfying \text{Hypothesis \ref{HYP5}} for $X$.
 We can write $X^M=(\Omega, \mathcal{M},\mathcal{M}_t,X_t,\theta_t, Q^x)$ where $Q^x$ is defined by \eqref{POS}. Clearly for any $t\geq 0$ we have
\[
	\{\omega\in\Omega:\lim_{n\rightarrow \infty}T_{E\setminus K_n(k_t \omega)}<\zeta(k_t \omega)\}\subset \{\omega\in\Omega:\lim_{n\rightarrow \infty}T_{E\setminus K_n(\omega)}<\zeta( \omega)\}
\]
where $k_t$ is the killing operator on $\Omega$. Therefore it follows that
\[\begin{aligned}
	Q^m(\lim_{n\rightarrow \infty}T_{E\setminus K_n}<\zeta)&=E^m\int_0^\infty (\lim_{n\rightarrow \infty}T_{E\setminus K_n}<\zeta)\circ k_t d(-M_t) \\
	&\leq E^m\int_0^\infty (\lim_{n\rightarrow \infty}T_{E\setminus K_n}<\zeta )d(-M_t)\\
	&=E^m(\lim_{n\rightarrow \infty}T_{E\setminus K_n}<\zeta)\\
	&=0.
\end{aligned}\]
That completes the proof.
 \end{proof}

In the rest of this section we shall discuss the sector condition. It will be outlined that under some mild condition, say \eqref{MCM}, $X^M$ still satisfies the sector condition and this condition is verified in Example~\ref{EMAP} for the typical pure-jump semi-Dirichlet forms and in Example~\ref{EXA2} for the multidimensional diffusion processes with jumps. In particular it is possible to characterize the associated semi-Dirichlet form of $X^M$.

First we assume that $M\in \text{MF}_{++}$. Then $E_M=E, m^*=m$ and $M$ satisfies \eqref{MFD}. Moreover it follows from \eqref{SLMD4} and \eqref{MME} that
 \[
\nu_{\bar{M}}(dxdy)=\nu_{[M]}(dxdy)= \Phi(x,y)\cdot \nu(dx dy)+\delta_y(dx)a(y)\mu_A(dy).
\]
Recall that $(P_t)_{t\geq 0}, (U^q)_{q\geq 0}$ (resp. $(Q_t)_{t\geq o},(V^\alpha)_{\alpha\geq 0}$) are the semigroup and resolvent of $X$ (resp. $X^M$). Since $M\in \text{MF}_{++}$ we have
\begin{equation}\label{UVF}
	U^qf=V^qf+U^q_{[M]}V^qf,
\end{equation}
\begin{equation}\label{UVF2}
    U^p_{[M]}f=U^{p+q}_{[M]}f+qU^{p+q}U^p_{[M]}f,
\end{equation}
for any $f\geq 0,p,q>0$ where
\[
	U^q_{[M]}V^qf(x):=E^x\int_0^\infty \text{e}^{-qt}V^qf(X_t)d[M]_t,\quad x\in E.
	\]
The following lemma is an extension of Lemma~3.7 of  \cite{JGY} but the idea of proof is different.

\begin{lemma}\label{UFV}
	It holds that
	\[
	V^1(pL^2(E,m))\subset \mathcal{F}
	\]
	 and for any $u\in \mathcal{F}, g\in pL^2(E,m)$,
\begin{equation}
	(u,g)_m=\mathcal{E}_\alpha (u, V^\alpha g)+\nu_{[M]}(\tilde{u}\otimes V^\alpha g)
\end{equation}
where $\tilde{u}\otimes V^\alpha g(x,y):=\tilde{u}(x)V^\alpha g(y)$ for any $x,y\in E$.
\end{lemma}
\begin{proof}
	Let $f\in pL^2(E,m)$ and we claim that $w=V^1f\in \mathcal{F}$. In fact it follows from \eqref{UVF} and \eqref{UVF2} that
	\[\begin{aligned}
	\beta(w, w-&\beta U^{\beta+1}w)_m\\= &\beta (w, U^1f-U^1_{[M]}V^1f-\beta U^{\beta+1}(U^1f-U^1_{[M]}V^1f))_m \\
	=&\beta (w, U^1f-U^1_{[M]}V^1f-\beta U^{\beta+1}(U^1f-U^1_{[M]}V^1f))_m\\
	=&\beta (w, U^1f-\beta U^{\beta+1}U^1f+\beta U^{\beta+1}U^1_{[M]}V^1f-U^1_{[M]}V^1f)_m\\
	=&\beta (w, U^{\beta+1}f-U^{\beta+1}_{[M]}V^1f)_m\\
	\leq& \beta (w, U^{\beta+1}f)_m \\
	\leq& ||w||_{L^2}\cdot ||f||_{L^2}.
	\end{aligned}\]	
Hence $\lim_{\beta\rightarrow \infty}\beta(w, w-\beta U^{\beta+1}w)_m<\infty$, in other words, $w\in \mathcal{F}$. Similarly we can deduce that $V^\alpha f=U^\alpha f-U^\alpha_{[M]}V^\alpha f\in \mathcal{F}$ whereas $U^\alpha f\in \mathcal{F}$. Thus $U^\alpha_{[M]}V^\alpha f\in \mathcal{F}$. Since $[M]$ has a decomposition \eqref{BMD3} with $S_M=\zeta$ it follows that
\[\begin{aligned}
	U^\alpha_{[M]}V&^\alpha f(x)\\=& E^x\sum_{s\leq t}\text{e}^{-\alpha s}V^\alpha f(X_s)\Phi(X_{s-},X_s)+E^x\int_0^\infty \text{e}^{-\alpha s}V^\alpha f(X_s)a(X_s)dA_s\\
	=& E^x\sum_{s\leq t}\text{e}^{-\alpha s}V^\alpha f(X_s)\Phi(X_{s-},X_s)+U^\alpha_A(V^\alpha f\cdot a)(x).
\end{aligned}\]
Clearly $U^\alpha_A(V^\alpha f\cdot a)$ is $\alpha$-excessive. Then from the fact
\[
	U^\alpha_A(V^\alpha f\cdot a)\leq U^\alpha_{[M]}V^\alpha f\in \mathcal{F}
	\]
	 and Theorem~2.16 of \cite{MOR} we obtain that $U^\alpha_A(V^\alpha f\cdot a)\in \mathcal{F}$. It follows from \text{Lemma \ref{RFP}} that for any $u\in \mathcal{F}$,
\begin{equation}\label{RFPF}
	\mathcal{E}_\alpha (u, U^\alpha_A(V^\alpha f\cdot a))=\mu_A(\tilde{u}\cdot V^\alpha f\cdot a)
\end{equation}
where $\mu_A$ is the smooth measure associated with $A$.
On the other hand it follows from \eqref{LS} that
\[\begin{aligned}
	E^x&\sum_{s\leq t}\text{e}^{-\alpha s}V^\alpha f(X_s)\Phi(X_{s-},X_s)\\&=E^x\int_0^t \text{e}^{-\alpha s} N(\Phi \cdot V^\alpha f)(X_s)dH_s\\&=U_H^\alpha (N(\Phi \cdot V^\alpha f))
\end{aligned}\]
where $(\Phi \cdot V^\alpha f)(x,y):=\Phi(x,y)V^\alpha f(y)$ for $x,y\in E$.  Similarly we have $U_H^\alpha (N(\Phi \cdot V^\alpha f))\in \mathcal{F}$ and
\begin{equation}\label{RFPF2}
\begin{aligned}
\mathcal{E}_\alpha (u, U_H^\alpha (N(\Phi \cdot V^\alpha f)))&=\mu_H(\tilde{u}\cdot N(\Phi \cdot V^\alpha f))\\&=\int \tilde{u}(x)V^\alpha f(y) \Phi(x,y)\nu(dxdy).
\end{aligned}\end{equation}
Thus it follows from \eqref{UVF}, \eqref{RFPF} and  \eqref{RFPF2} that
\[\begin{aligned}
		(u,f)_m&=\mathcal{E}_\alpha (u, U^\alpha f)\\&=\mathcal{E}_\alpha (u, V^\alpha f)+\mathcal{E}_\alpha (u,U^\alpha_{[M]}V^\alpha f)\\
	&=	\mathcal{E}_\alpha (u, V^\alpha f)+\mu_A(\tilde{u}\cdot V^\alpha f\cdot a)+\int \tilde{u}(x)V^\alpha f(y) \Phi(x,y)\nu(dxdy)\\&=	\mathcal{E}_\alpha (u, V^\alpha f)+ \nu_{[M]}(\tilde{u}\otimes V^\alpha f) .
\end{aligned}\]
That completes the proof. 
\end{proof}

Our main results on the sector condition related to $X^M$ are as follows. Note that the lower bounded semi-Dirichlet forms are  introduced in Appendix~\ref{SDF}.

\begin{theorem}\label{FKT}
	Let $X$ be a right Markov process satisfying Hypothesis \ref{HYP3}, \ref{HYP1} and  \ref{HYP5}, $(\mathcal{E},\mathcal{F})$ its associated semi-Dirichlet form on $L^2(E,m)$ and $M\in \text{MF}_{++}$. Assume that there exist two  constants $c>\frac{1}{4}$ and $\lambda_0\geq 0$ such that
\begin{equation}\label{MCM}
	\mathcal{E}_{\lambda_0}(u,u)\geq c\int (u(x)-u(y))^2 \nu_{[M]}(dxdy)	
\end{equation}
for any $u\in \mathcal{F}$. Then the subprocess $X^M=(X,M)$ of $X$ satisfies the sector condition and its associated semi-Dirichlet form $(\mathcal{E}^M,\mathcal{F}^M)$ are
\begin{equation}\label{EFMD}
\begin{aligned}
		\mathcal{F}^M&=\mathcal{F}\cap L^2(E,\nu^1_{[M]}+\nu^2_{[M]}),\\
		\mathcal{E}^M(u,v)&=\mathcal{E}(u,v)+\nu_{[M]}(u\otimes v) \quad u,v\in \mathcal{F}^M
\end{aligned}\end{equation}
where $\nu_{[M]}^1$ and $\nu_{[M]}^2$ are the left and right marginal measures of $\nu_{[M]}$. In particular $(\mathcal{E}^M,\mathcal{F}^M)$ is quasi-regular and $X^M$ is properly associated with it.
\end{theorem}
\begin{proof}
	First we assert that under the condition  \eqref{MCM} it holds that
 \[
 	V^1(L^2(E,m))\subset \mathcal{F}^M.
 	\]
  For any $\lambda>\lambda_0$ and $g\in pL^2(E,m)$ it follows from Lemma~\ref{UFV} that $u:=V^{\lambda} g\in \mathcal{F}$. On the other hand clearly $u(x)^2+u(y)^2\geq \frac{1}{2}(u(x)-u(y))^2$.
	Then if $c\leq \frac{1}{2}$ it follows from \eqref{MCM} that
\[\begin{aligned}
    (u,g)_m=&\mathcal{E}_{\lambda}(u,u)+\nu_{[M]}(u\otimes u)  \\
            =&\mathcal{E}_{\lambda}(u,u)-\frac{1}{2}\int (u(x)-u(y))^2 \nu_{[M]}(dxdy)\\&\quad +\frac{1}{2}\int (u(x)^2+u(y)^2)\nu_{[M]}(dxdy) \\
	=&\mathcal{E}_{\lambda}(u,u)-\frac{1}{2}\int (u(x)-u(y))^2 \nu_{[M]}(dxdy)\\&\quad+\left((1-2c)+(2c-\frac{1}{2})\right)\int (u(x)^2+u(y)^2)\nu_{[M]}(dxdy).
	\end{aligned}\]
Hence we have
	\[\begin{aligned}
    (u,g)_m       & \geq \mathcal{E}_{\lambda}(u,u)-\frac{1}{2}\int (u(x)-u(y))^2 \nu_{[M]}(dxdy)\\&\quad+(\frac{1}{2}-c)\int (u(x)-u(y))^2 \nu_{[M]}(dxdy))\\&\quad
			+(2c-\frac{1}{2})\int (u(x)^2+u(y)^2)\nu_{[M]}(dxdy) \\
	&=\mathcal{E}_{\lambda}(u,u)-c\int (u(x)-u(y))^2 \nu_{[M]}(dxdy)\\&\quad +(2c-\frac{1}{2})\int (u(x)^2+u(y)^2)\nu_{[M]}(dxdy)\\
	&\geq (2c-\frac{1}{2})\int (u(x)^2+u(y)^2)\nu_{[M]}(dxdy).
\end{aligned}\]
Similarly if $c>\frac{1}{2}$ we can deduce that
\[
	(u,g)_m\geq \frac{1}{2}\int (u(x)^2+u(y)^2)\nu_{[M]}(dxdy)
	\] and thus
\begin{equation}\label{MEL}
	\mathcal{E}^M_{\lambda}(u,u)\geq\left( (2c-\frac{1}{2})\wedge \frac{1}{2}\right)\int (u(x)^2+u(y)^2)\nu_{[M]}(dxdy).
\end{equation}
Note that \eqref{MEL} still holds for any $u\in\mathcal{F}^M$.
Moreover since $(V^\lambda g,g)_m<\infty$ it follows that $V^\lambda g\in L^2(E,\nu^1_{[M]}+\nu^2_{[M]})$ and in particular  \[V^1(L^2(E,m))=V^{\lambda}(L^2(E,m))\subset  \mathcal{F}^M.\] Hence $\mathcal{F}^M$ is dense in $L^2(E,m)$ and $\mathcal{E}^M_{\lambda_0}(u,u)\geq 0$ for any $u\in \mathcal{F}^M$.

Secondly we shall prove that $(\mathcal{E}^M,\mathcal{F}^M)$ is a lower bounded closed form.
Let $\lambda>\lambda_0$ be a constant and $\{u_n:n\geq 1\}\subset \mathcal{F}^M$ an $\tilde{\mathcal{E}}^M_\lambda$-Cauchy sequence. It follows from \eqref{MEL} that the sequence $\{u_n:n\geq 1\}$ ($\subset \mathcal{F}$) is also $L^2(E,\nu^1_{[M]}+\nu^2_{[M]})$-Cauchy. Thus there exists a function $u\in L^2(E,\nu^1_{[M]}+\nu^2_{[M]})$ such that $u_n\rightarrow u$ in $L^2(E,\nu^1_{[M]}+\nu^2_{[M]})$. In particular there exists a subsequence $\{u_{n_k}:k\geq 1\}$ of $\{u_n:n\geq 1\}$ such that $u_{n_k}\rightarrow u$, $\nu^1_{[M]}+\nu^2_{[M]}$-a.e. Since
\begin{equation}\label{UVM12}
    |\nu_{[M]}(v\otimes v)|\leq \frac{1}{2}\int (v(x))^2(\nu_{[M]}^1(dx)+\nu_{[M]}^2(dx))
\end{equation}
for any $v\in L^2(E,\nu^1_{[M]}+\nu^2_{[M]})$, it follows that $\nu_{[M]}((u_n-u_m)\otimes (u_n-u_m))\rightarrow 0$ and $\{u_n:n\geq 1\}$ is $\tilde{\mathcal{E}}_\lambda$-Cauchy. Therefore we can choose a $u'\in \mathcal{F}$ such that $\mathcal{E}_\lambda(u_n-u',u_n-u')\rightarrow 0$ as $n\rightarrow \infty$. In particular $\mathcal{E}_\lambda(u_{n_k}-u',u_{n_k}-u')\rightarrow 0$ as $k\rightarrow \infty$. Then there exists a subsequence of $\{u_{n_k}\}$, denoted by $\{u_l\}$, such that $u_l\rightarrow u'$ q.e. However $\nu^1_{[M]},\nu^2_{[M]}$ charge no $m$-polar sets and hence $u_l\rightarrow u', \nu^1_{[M]}+\nu^2_{[M]}$-a.e. It follows that $u=u'$. Moreover we can deduce that $u\in \mathcal{F}^M$ and $\mathcal{E}^M_\lambda(u_n-u,u_n-u)\rightarrow 0$ as $n\rightarrow \infty$. 

Now we claim that there exists a constant $K_\lambda >0 $ such that
\begin{equation}\label{EMSC}
    |\mathcal{E}^M_\lambda (u,v)|\leq K_\lambda\cdot \mathcal{E}^M_\lambda(u,u)^{\frac{1}{2}}\cdot \mathcal{E}^M_\lambda(v,v)^{\frac{1}{2}}
\end{equation}
for any $u,v\in \mathcal{F}^M$. In fact it follows from \eqref{MEL} and \eqref{UVM12} that
\[\begin{aligned}
	(1+\frac{1}{2c'})\mathcal{E}^M_\lambda(u,u)&\geq\mathcal{E}^M_\lambda(u,u)+\frac{1}{2}\int (u(x)^2+u(y)^2)\nu_{[M]}(dxdy)\\&\geq \mathcal{E}_\lambda(u,u)
\end{aligned}\]
where $c'=(2c-\frac{1}{2})\wedge \frac{1}{2}$.  Since $(\mathcal{E},\mathcal{F})$ satisfies the sector condition there exists a constant $K^1_\lambda >0$ such that
\[\begin{aligned}
    |\mathcal{E}_\lambda (u,v)|&\leq K^1_\lambda\cdot \mathcal{E}_\lambda(u,u)^{\frac{1}{2}}\cdot \mathcal{E}_\lambda(v,v)^{\frac{1}{2}}\\&\leq K^1_\lambda\cdot c''\cdot \mathcal{E}^M_\lambda(u,u)^{\frac{1}{2}}\cdot \mathcal{E}^M_\lambda(v,v)^{\frac{1}{2}}
\end{aligned}\]
where $c''=1+\frac{1}{2c'}$.
By Cauchy-Schwarz inequality and \eqref{MEL} we obtain
\[\begin{aligned}
	|\nu_{[M]}(u\otimes v)|&\leq (\int u(x)^2\nu_{[M]}(dxdy))^{\frac{1}{2}}\cdot (\int v(y)^2\nu_{[M]}(dxdy))^{\frac{1}{2}}\\&\leq \frac{1}{c'}\mathcal{E}^M_\lambda(u,u)^{\frac{1}{2}}\cdot \mathcal{E}^M_\lambda(v,v)^{\frac{1}{2}}.
\end{aligned}\]
Then \eqref{EMSC} holds with the parameter $K_\lambda:=K^1_\lambda\cdot c''+\frac{1}{c'}$.

Since we have proved that $(\mathcal{E}^M,\mathcal{F}^M)$ is a lower bounded closed form,  there exists a unique strongly continuous resolvent $(G_\alpha)_{\alpha\geq 0}$ such that
\[
	\mathcal{E}^M_\lambda(u, G_\lambda f)=(u,f)_m,\quad  u\in \mathcal{F}^M, f\in L^2(E,m), \lambda>\lambda_0.
\]
However from Lemma~\ref{UFV} we can see that
\[
	\mathcal{E}^M_\lambda(u, V^\lambda f)=(u,f)_m,\quad  u\in \mathcal{F}^M, f\in L^2(E,m), \lambda>\lambda_0.
\]
Hence $G_\alpha f=V^\alpha f$ for any $\alpha >0, f\in L^2(E,m)$. It follows from Lemma \ref{PQC} and \ref{LBTN} that $(\mathcal{E}^M,\mathcal{F}^M)$ is non-negative.

Finally we only need to prove that $(\mathcal{E}^M,\mathcal{F}^M)$ has the semi-Dirichlet property, or equivalently, for any $f\in L^2(E,m)$ such that $0\leq f\leq 1$ it follows that $0\leq \alpha V^\alpha f\leq 1$ for any $\alpha>0$. This fact is apparent because $(V^\alpha)_{\alpha>0}$ is the resolvent of $X^M$. Hence we have already proved that $(\mathcal{E}^M,\mathcal{F}^M)$ is a semi-Dirichlet form.  It follows from Lemma \ref{PQC} and \ref{HYLM4} that $X^M$ satisfies Hypothesis \ref{HYP3}, \ref{HYP1} and \ref{HYP5}. In particular the semi-Dirichlet form $(\mathcal{E}^M,\mathcal{F}^M)$ is quasi-regular and $X^M$ is properly associated with $(\mathcal{E}^M,\mathcal{F}^M)$.  
\end{proof}

\begin{remark}\label{RM2}
(1) If $M$ is continuous or $X$ is continuous, i.e. $\Phi\equiv 0$, then $M_t=\text{e}^{-A^*_t}$ where $A^*_t=\int_0^t a(X_s)dA_s\in \text{PCAF}$ and in particular \eqref{MCM} is satisfied. Let $\mu_{A^*}$ be the corresponding smooth measure of $A^*$. Then $\nu_{[M]}(dxdy)=\delta_x(dy)\mu_{A^*}(dx)$ and the associated semi-Dirichlet form $(\mathcal{E}^{A^*},\mathcal{F}^{A^*})$ of $(X,M)$ is
\[\begin{aligned}
	\mathcal{F}^{A^*}&= \mathcal{F}\cap L^2(E,\mu_{A^*}),  \\
	\mathcal{E}^{A^*}(u,v)&=\mathcal{E}(u,v)+\int u(x)v(x)\mu_{A^*}(dx)\quad  u,v\in \mathcal{F}^{A^*}.
\end{aligned}\]
Hence $(\mathcal{E}^{A^*},\mathcal{F}^{A^*})$ is exactly the perburbed Dirichlet form of $(\mathcal{E},\mathcal{F})$ by smooth measure $\mu_{A^*}$. This has been discussed in \S4.3 of  \cite{OY} for the cases that the smooth measure is Radon. The general cases in the context of the semi-Dirichlet forms are similar to those of the non-symmetric Dirichlet forms, see IV\S4(c) of  \cite{MR}. In \text{\S\ref{KAS}} we shall also make some characterizations to the smooth measures in perturbations.

(2) If $m$ is excessive, equivalently $X$ has a dual Markov process relative to $m$ or $(\mathcal{E},\mathcal{F})$ is a non-symmetric Dirichlet form on $L^2(E,m)$, then $\mathcal{E}$ has a Beurling-Deny type decomposition on the diagonal
\[
    \mathcal{E}(u,u)=\mathcal{E}^{(c)}(u,u)+\frac{1}{2}\int (u(x)-u(y))^2 \nu(dxdy)\quad  u\in \mathcal{F},
\]
where the non-negative form $\mathcal{E}^{(c)}$ is the continuous part of $X$ and $\nu$ is exactly the canonical measure of $X$. Hence the condition \eqref{MCM} is satisfied with the parameters $\lambda_0=0,c=\frac{1}{2}$ because $0\leq \Phi < 1$.

(3) For the general semi-Dirichlet forms Oshima also gives a decomposition (see Theorem~5.2.1 of  \cite{OY}) for the regular semi-Dirichlet form:
\begin{equation}\label{EQMEUV}
    \mathcal{E}(u,v)=\mathcal{E}^{(c)}(u,v)+\mathcal{E}^{(j)}(u,v)+\int u(x)v(x)k(dx),
\end{equation}
where $k$ is the killing measure of $X$ and the non-local part $\mathcal{E}^{(j)}$ of the decomposition  is given by:
\begin{equation}\label{EQMEJ}
\begin{aligned}
   \mathcal{E}^{(j)}(u,v):= &\frac{1}{2}\int (u(x)-u(y))(v(x)-v(y))\nu(dxdy)\\&-\frac{1}{2}\int (v(y)-v(x))u(x)(\nu(dxdy)-\nu(dydx))
\end{aligned}\end{equation}
for any $u,v\in \mathcal{F}\cap C_c(E)$. 
Since the canonical measure of $X^M$ is $\nu^M(dxdy)=(1-\Phi(x,y))\nu(dxdy)$, it holds that
\[\begin{aligned}
    \mathcal{E}^M(u,v)=&\mathcal{E}^{(c)}(u,v)+\frac{1}{2}\int (u(x)-u(y))(v(x)-v(y))\nu^M(dxdy) \\
                &-\frac{1}{2}\int (v(y)-v(x))u(x)(\nu^M(dxdy)-\nu^M(dydx))\\
                    &+\int u(x)v(x)(k+\lambda)(dx)
\end{aligned}\]
where $\mathcal{E}^{(c)}$ is the semi-strongly local part in the decomposition of $\mathcal{E}$, $k$ is the killing measure of $X$ and $\lambda(dx)=\nu^1_{[M]}(dx)$ is the Revuz measure of $[M]$ in \eqref{RM}. Hence roughly speaking, the killing transform by $M\in \text{MF}(X)$ is essentially to multiply the canonical measure of $X$ by $1-\Phi$ and to add the (left) Revuz measure of $[M]$ to its killing measure.
\end{remark}

Now let $(\mathcal{E}^{\alpha_0},\mathcal{F}^{\alpha_0})$ be a quasi-regular lower bounded semi-Dirichlet form with the parameter $\alpha_0\geq 0$. It always has an associated Markov process denoted by $X^{\alpha_0}$, see \S3.3 of  \cite{OY} for the regular cases on a locally compact separable metric space. By the quasi-homeomorphism method appeared in  \cite{CMR} or  \cite{HMS}, the existence of $X^{\alpha_0}$ can be extended to quasi-regular cases on a Hausdorff topological space. Clearly the semigroup $(P^{\alpha_0}_t)_{t\geq 0}$ of $X^{\alpha_0}$ does not satisfy \text{Hypothesis \ref{HYP3}} (if $\alpha_0>0$) but $(\text{e}^{-\alpha_0 t}P^{\alpha_0}_t)_{t\geq 0}$ does. However
in the proof of \text{Theorem \ref{FKT}}, we can find that all other properties are kept if we replace the non-negative property by the lower boundedness assumption. In  other words, there is no essential difference between the non-negative semi-Dirichlet forms and lower bounded semi-Dirichlet forms when discussing the killing transforms. Thus we have the following theorem. It proof is similar to \text{Theorem \ref{FKT}}, so we omit it.

\begin{theorem}\label{THMLX}
    Let $(\mathcal{E}^{\alpha_0},\mathcal{F}^{\alpha_0})$ be a lower bounded semi-Dirichlet form with the parameter $\alpha_0\geq 0$ on $L^2(E,m)$ and $X^{\alpha_0}$ its associated Markov process. Assume that the semigroup $(P^{\alpha_0}_t)_{t\geq 0}$ of $X^{\alpha_0}$ satisfies that \text{Hypothesis \ref{HYP3},  \ref{HYP1}} and \text{\ref{HYP5}} hold for the semigroup $(\text{e}^{-\alpha_0 t}P^{\alpha_0}_t)_{t\geq 0}$. Further let $M\in \text{MF}_{++}(X^{\alpha_0})$ and $\nu_{[M]}$ the bivariate Revuz measure of the  Stieltjes logarithm $[M]$. If there exist two constants $\lambda_0\geq \alpha_0,c>\frac{1}{4}$ such that
\[
	\mathcal{E}_{\lambda_0}(u,u)\geq c\int (u(x)-u(y))^2 \nu_{[M]}(dxdy) 	
\]
for any $u\in \mathcal{F}$,
then the subprocess $X^{\alpha_0,M}:=(X^{\alpha_0},M)$ of $X^{\alpha_0}$ satisfies the sector condition and its properly associated quasi-regular semi-Dirichlet form $(\mathcal{E}^{\alpha_0,M},\mathcal{F}^{\alpha_0,M})$ is lower bounded with the parameter $\alpha_0$ and given by:
\begin{equation}\label{FMA}
\begin{aligned}
		\mathcal{F}^{\alpha_0,M}&=\mathcal{F}^{\alpha_0}\cap L^2(E,\nu^1_{[M]}+\nu^2_{[M]}),\\
		\mathcal{E}^{\alpha_0,M}(u,v)&=\mathcal{E}^{\alpha_0}(u,v)+\nu_{[M]}(u\otimes v) \quad u,v\in \mathcal{F}^{\alpha_0, M},
\end{aligned}\end{equation}
where $\nu_{[M]}^1$ and $\nu_{[M]}^2$ are the left and right marginal measures of $\nu_{[M]}$.
\end{theorem}

The following theorem is an extension of Theorem~\ref{FKT} to the general multiplicative functionals. The proof is similar to that of Theorem~4.1 of  \cite{JGY} whereas it also needs some new techniques outlined in Theorem~\ref{EXIS} to deal with the absence of  weak duality assumption. We put its proof into Appendix~\ref{PROOF}. Similarly we can also obtain an extension of Theorem~\ref{THMLX} to  general multiplicative functionals and the main assumption \eqref{EUBM} remains. Due to space limiations, we won't go into details here.

\begin{theorem}\label{GCT}
    Let $X$ be a right Markov process satisfying \text{Hypothesis \ref{HYP3}}, \text{\ref{HYP1}}  and \ref{HYP5} and $(\mathcal{E},\mathcal{F})$  its associated semi-Dirichlet form on $L^2(E,m)$. Fix $M\in\text{MF}(X)$ and $m^*:=1_{E_M}\cdot m$. Assume that there exist two constants $\lambda_0\geq 0,c>\frac{1}{4}$ such that
    \begin{equation}\label{EUBM}
        \mathcal{E}_{\lambda_0}(u,u)\geq c\int (u(x)-u(y))^2\nu_{\bar{M}}(dxdy)
    \end{equation}
for any $u\in \mathcal{F}$   where $\nu_{\bar{M}}$ is the bivariate Revuz measure of $\bar{M}$. Then the subprocess $X^M=(X,M)$ on $E_M$ satisfies the sector condition and its properly associated quasi-regular semi-Dirichlet form $(\mathcal{E}^M,\mathcal{F}^M)$ on $L^2(E_M,m^*)$ is given by
    \begin{equation}\label{EFMD2}
    \begin{aligned}
        \mathcal{F}^M&=\mathcal{F}_{E_M}\cap L^2(E_M,\nu_{\bar{M}}^1+\nu_{\bar{M}}^2);   \\
        \mathcal{E}^M(u,v)&=\mathcal{E}(u,v)+\nu_{\bar{M}}(u\otimes v)\quad u,v\in \mathcal{F}^M,
    \end{aligned}
    \end{equation}
    where $\mathcal{F}_{E_M}:=\{u\in \mathcal{F}:u=0\;\text{q.e. on }E^c_M\}$ is the restricted space of $\mathcal{F}$ on $E_M$.
\end{theorem}

Note that every quasi-regular semi-Dirichlet form is quasi-homeomorphic to a regular Dirichlet space on a locally compact separable metric space. We shall discuss the condition \eqref{EUBM} in the context of regular semi-Dirichlet forms. Assume that $E$ is such a metric space and $m$ is a Radon measure on $E$. Recall that if $(\mathcal{E,F})$ is a regular semi-Dirichlet form on $L^2(E,m)$ with a core $\mathcal{C}$ and Assumption~(J) in \S5.2 of  \cite{OY} holds, then $(\mathcal{E,F})$ has a Beurling-Deny type decomposition, i.e. \eqref{EQMEUV} holds for any $u,v\in \mathcal{C}$. We refer the details to Theorem~5.2.1 of \cite{OY}. Let $\nu$ be the canonical measure of the associated Hunt process $X$ of $(\mathcal{E,F})$. Assume further that the form 
\begin{equation}\label{MQI}
\begin{aligned}
    \mathcal{Q}(u,v)&:=\int (u(x)-u(y))(v(x)-v(y))\nu(dxdy),  \\
    \mathcal{F}_{\mathcal{Q}}&:=\{u\in L^2(E,m):\mathcal{Q}(u,u)<\infty\}
\end{aligned}\end{equation}
is a symmetric Dirichlet form on $L^2(E,m)$. For example, $\nu$ is absolutely continuous with respect to $m\times m$, i.e.
	\begin{equation}\label{EQNUD}
		\nu(dxdy)=j(x,y)m(dx)m(dy)
	\end{equation}
for some non-negative function $j$ on $E\times E\setminus d$, see Lemma~1.5.6 of \cite{OY}. 

\begin{lemma}\label{LM7}
	Assume that there exist two constants $\lambda_0\geq 0$ and $c>1/4$ such that for any $u\in \mathcal{C}$,
	\begin{equation}\label{EQMEL}
	 \mathcal{E}_{\lambda_0}(u,u)\geq c\mathcal{Q}(u,u).
\end{equation}
Then for any $M\in \text{MF}$, the condition \eqref{EUBM} holds for any $u\in \mathcal{F}$ with the same parameters $\lambda_0$ and $c$.
\end{lemma}
\begin{proof}
	Without loss of generality we may assume that $M\in \text{MF}_+$.  Then it follows from Proposition~\ref{USC} that $\nu_{\bar{M}}|_{E\times E\setminus d}\leq \nu$. Thus we only need to prove that \eqref{EQMEL} holds for any $u\in \mathcal{F}$. To this end take a sequence $\{u_n:n\geq 1\}\subset \mathcal{C}$ such that 
$\mathcal{E}_{\lambda_0}(u_n-u,u_n-u)\rightarrow 0$ as $n\rightarrow\infty$ for some $\lambda_0$ large enough. In particular $\{u_n:n\geq 1\}$ is $\tilde{\mathcal{E}}_{\lambda_0}$-Cauchy and hence also $\mathcal{Q}_1$-Cauchy. Since $u_n\rightarrow u$ in $L^2(E,m)$ we can deduce that $u\in \mathcal{F}_\mathcal{Q}$ and $\mathcal{Q}_1(u_n-u,u_n-u)\rightarrow 0$ as $n\rightarrow \infty$. Clearly \eqref{EQMEL} holds for any $u_n$. By letting $n\rightarrow \infty$ we have \eqref{EQMEL} also holds for $u$.  
\end{proof}

In the end of this section , we shall present two examples of typical semi-Dirichlet forms which 
are introduced by other researchers and try to illustrate that the condition \eqref{EUBM} is not so awkward.
In the first example it will be seen that the (lower bounded) jump-type semi-Dirichlet form under the assumption \eqref{QVC} always satisfies the condition \eqref{EUBM}. In particular \eqref{QVC} is a typical sufficient condition to obtain the sector condition of jump-type semi-Dirichlet form, see \cite{FUT}, \cite{OY} and  \cite{SW}.

\begin{example}\label{EMAP}
Let $\nu$ be a $\sigma$-finite positive measure on $E\times E\setminus d$ and assume that the family $C_c^{\text{lip}}(E)$ of all Lipschitz continuous functions with compact support on $E$ is a subspace of $\mathcal{F}_\mathcal{Q}$ where  $(\mathcal{Q},\mathcal{F}_\mathcal{Q})$ defined by \eqref{MQI} is a symmetric Dirichlet form on $L^2(E,m)$.
Moreover for any $u,v\in C_c^{\text{lip}}(E)$ define another form
\[
    \mathcal{E}(u,v):=\frac{1}{2}\mathcal{Q}(u,v)+\frac{1}{2}\int (v(x)-v(y))u(y)(\nu(dxdy)-\nu(dydx)).
\]
Suppose that the following assumption
\begin{equation}\label{QVC}
    \bigg|\int (v(x)-v(y))u(y)(\nu(dxdy)-\nu(dydx))\bigg|\leq K||u||_{L^2(E,m)}\cdot \sqrt{\mathcal{Q}(v,v)}
\end{equation}
holds for some constant $K$ which is independent of $u,v\in C_c^{\text{lip}}(E)$. Then the domain of the form $\mathcal{E}$ can be extended to some dense subspace $\mathcal{F}$ of $L^2(E,m)$ and $(\mathcal{E},\mathcal{F})$ is a lower bounded semi-Dirichlet form on $L^2(E,m)$.  The typical examples of  pure-jump type Markov processes which satisfy all the above conditions but not the duality assumption are the stable-like processes, i.e. $E=R^d$, $\nu(dxdy)=j(x,y)dxdy$ where
\[
	j(x,y):=w(x)|x-y|^{-d-\alpha(x)}
\]
and $w(x), \alpha(x)$ satisfy (5.1) and (5.2) of  \cite{FUT}. For more details, see  \cite{FUT}, \cite{SW} and \S1.5.2 of  \cite{OY}.

Let $X$ be the associated Hunt process of $(\mathcal{E},\mathcal{F})$. Clearly $\nu$ is exactly the canonical measure of $X$. We assert that $(\mathcal{E,F})$ satisfies \eqref{EQMEL} for any $u\in C_c^{\text{lip}}(E)$. 
To this end fix a constant $\frac{1}{4}<c<\frac{1}{2}$, $\lambda_0:= K^2/(8-16c)$ and let
 \begin{equation}\label{EQAUV}
 	\mathcal{A}(u,v):=\int (v(x)-v(y))u(y)(\nu(dxdy)-\nu(dydx))\end{equation} for any $u,v\in C_c^{\text{lip}}(E)$.
By \eqref{QVC} and H\"older inequality we have 
\[\begin{aligned}
	 \mathcal{E}_{\lambda_0}&(u,u)\\&=c\mathcal{Q}(u,u)+(\frac{1}{2}-c)\mathcal{Q}(u,u)+\lambda_0(u,u)_m+\frac{1}{2}\mathcal{A}(u,u)\\
					&\geq  c\mathcal{Q}(u,u)+(\frac{1}{2}-c)\mathcal{Q}(u,u)+\lambda_0(u,u)_m-\frac{1}{2}K||u||_{L^2(E,m)}\cdot \sqrt{\mathcal{Q}(u,u)} \\
					&\geq  c\mathcal{Q}(u,u)
\end{aligned}\]
for any $u\in C_c^{\text{lip}}(E)$.
It follows from Lemma~\ref{LM7} that for any $M\in \text{MF}$, \eqref{EUBM} holds for any $u\in \mathcal{F}$ with the above parameters $\lambda_0$ and $c$. 
In particular $X^M$ satisfies the sector condition and its associated lower bounded semi-Dirichlet form can be obtained similarly to Theorem~\ref{GCT}.
\end{example}

We use the same notation $\mathcal{A}$ as \eqref{EQAUV} to denote the antisymmtric part of $\mathcal{E}^{(j)}$ in the Beurling-Deny type decomposition \eqref{EQMEUV} of $(\mathcal{E,F})$.

\begin{proposition}\label{PROP9}
	Assume that \eqref{QVC} holds for any $u,v\in \mathcal{C}$ with some constant $K$ which is independent of $u,v$ and there exist two constants $c_1, \beta$ such that $0\leq c_1<1/4,\beta\geq 0$ and
\begin{equation}\label{EQMEC}
	\mathcal{E}^{(c)}(u,u)+c_1\mathcal{Q}(u,u) +\int u(x)^2k(dx)+\beta \int u(x)^2m(dx)\geq 0
\end{equation}
 for any $u\in\mathcal{C}$. Then for any $M\in \text{MF}$, the condition \eqref{EUBM} holds for any $u\in \mathcal{F}$ with the parameters $c$ and $\lambda_0$ such that $1/4<c<1/2-c_1$ and $\lambda_0=K^2/(8-16c-16c_1)+\beta$.
\end{proposition}
\begin{proof}
Fix two constants  $c$ and $\lambda_0$ as above.
For any $u\in \mathcal{C}$ we have
\[\begin{aligned}
	\mathcal{E}_{\lambda_0}(u,u)=\mathcal{E}&^{(c)}(u,u)+c_1\mathcal{Q}(u,u) +\int u(x)^2k(dx)+\beta \int u(x)^2m(dx) \\
		&+(\frac{1}{2}-c-c_1)\mathcal{Q}(u,u)+\frac{1}{2}\mathcal{A}(u,u)+\frac{K^2}{8-16c-16c_1}(u,u)_m\\
		&+c\mathcal{Q}(u,u).
\end{aligned}\]
It follows from \eqref{QVC} and \eqref{EQMEC} that $\mathcal{E}_{\lambda_0}(u,u)\geq c\mathcal{Q}(u,u)$ for any $u\in \mathcal{C}$.  By Lemma~\ref{LM7} we can obtain the conclusion. 
\end{proof}

Note that the semi-local part, i.e. the first and third terms in the right side of \eqref{EQMEUV}, is not necessarily non-negative or lower bounded. But clearly we have the following corollary of Proposition~\ref{PROP9}.

\begin{corollary}\label{COR3}
Assume that \eqref{QVC} holds for any $u,v\in \mathcal{C}$ with some constant $K$ which is independent of $u,v$.	If the semi-local part of $(\mathcal{E,F})$ is lower bounded, then the condition \eqref{EUBM} holds for any $M\in\text{MF}$.
\end{corollary}

The second example is taken from \cite{UT} in which the author characterizes the associated (lower bounded) semi-Dirichlet forms of multidimensional diffusion processes with jumps. These semi-Dirichlet forms satisfy the Beurling-Deny type decomposition.
In the following example we will illustrate that any (lower bounded) semi-Dirichlet form outlined in  \cite{UT} satisfies the condition \eqref{EUBM} for any $M\in \text{MF}$. In particular its killing transform by any MF always keeps the sector condition.

\begin{example}\label{EXA2}
The authors of \cite{UT} considered the following second partial differential operator with a non-local part:
\[
\begin{aligned}
	\mathcal{L}u(x):=&\mathcal{L}_cu(x)+\mathcal{L}_ju(x)\\
		=&\frac{1}{2}\sum_{i,j=1}^d\frac{\partial}{\partial x_i}\left(a_{ij}(x)\frac{\partial}{\partial x_j}\right)u(x)-\sum_{i=1}^d b_i(x)\frac{\partial}{\partial x_i}u(x)-c(x)u(x)\\
		&\quad +\lim_{n\rightarrow \infty}\frac{1}{2}\int_{|x-y|>1/n}\left(u(y)-u(x)\right)k(x,y)dy,\quad x\in G,
\end{aligned}\]	
where $a_{ij},b_i$ and $c$ are measurable functions defined on an open set $G$ of $\mathbf{R}^d$ for $i,j=1,2,\cdots,d$ and $k(x,y)$ is a measurable function defined on $G\times G\setminus \{(x,x):x\in G\}$. Under some appropriate conditions its associated semi-Dirichlet form can be written as
\[
	\eta(u,v)=\eta^{(c)}(u,v)+\eta^{(j)}(u,v),
\]
for any $u,v\in C_c^1(G)$ where
\[
\begin{aligned}
	\eta^{(c)}(u,v)=\frac{1}{2}&\sum_{i,j=1}^d\int_G a_{ij}(x)\frac{\partial u}{\partial x_i}(x)\frac{\partial v}{\partial x_j}(x)dx +\sum_{i=1}^d\int_G b_i(x)v(x)\frac{\partial u}{\partial x_i}(x)dx\\
	&+\int_G u(x)v(x)c(x)dx,
\end{aligned}\]	
and the non-local part $\eta^{(j)}$ is similar to \eqref{EQMEJ} by replacing $\nu$ with $k(x,y)dxdy$.

For the uniformly elliptic case, i.e. $(a_{ij})_{1\leq i,j\leq d}$ satisfies the uniformly elliptic condition, under some other assumptions (say (D.1)-(D.3) and (J.1) (J.2) of \cite{UT}) the form $\eta$ can be extended from $C_c^1(G)\times C_c^1(G)$ to $\mathcal{F}\times \mathcal{F}$ to be  a regular lower bounded semi-Dirichlet form $(\eta, \mathcal{F})$ on $L^2(G)$, see Theorem~3.1 of \cite{UT}. In particular \eqref{QVC} is satisfied for $\nu=k(x,y)dxdy, u,v\in C_c^1(G)$ (see (2.11) of \cite{FUT}) and  the semi-local part $\eta^{(c)}$ is lower bounded, see the proof of Proposition~3.2 of \cite{UT}. Thus it follows from Corollary~\ref{COR3} that the condition \eqref{EUBM} holds for any $M\in\text{MF}$.

For the degenerate case on $G=\mathbf{R}^d$, i.e. $(a_{ij})_{1\leq i,j\leq d}$ is only non-negative definite, under some conditions $\eta$ can also be extended from $C_c^1(\mathbf{R}^d)\times C_c^1(\mathbf{R}^d)$ to $\mathcal{F}\times \mathcal{F}$ to be  a regular lower bounded semi-Dirichlet form $(\eta, \mathcal{F})$ on $L^2(\mathbf{R}^d)$, see Theorem~4.1 of \cite{UT}. In particular \eqref{QVC} is also satisfied and there exists a constant $\beta\geq 0$ such that
\[
	\eta^{(c)}(u,u)+\frac{1}{16}\int \left(u(x)-u(y)\right)^2k(x,y)dxdy+\beta||u||^2_{L^2}\geq 0
\]
for any $u\in C_c^1(\mathbf{R}^d)$, see the first inequality in the proof of Theorem~4.1 of \cite{UT}. From Proposition~\ref{PROP9} we can deduce that the condition \eqref{EUBM} also holds for any $M\in\text{MF}$.

We refer more specific examples to \S6 of \cite{UT}.
\end{example}

\section{Killing and subordination}\label{KAS}

In this section we shall extend the results of  \cite{JGY3}, which states that killing transform in Markoc processes is equivalent to subordination in Dirichlet form, to the semi-Dirichlet forms. Since the idea of proof is essentially the same, we only state the results and omit the proofs here.

 Let $X$ be a right process on $E$ satisfying \text{Hypothesis \ref{HYP3}, \ref{HYP1}}  and \text{ \ref{HYP5}} and $(\mathcal{E},\mathcal{F})$ its associated quasi-regular semi-Dirichlet form on $L^2(E,m)$. Define a class of multiplicative functionals of $X$ by
\begin{equation}
    \text{MF}^*_+ :=\{M\in \text{MF}_+: M \text{ satisfies \eqref{EUBM}}  \}.
\end{equation}
Note that if $m$ is excessive, then $\text{MF}_+^*=\text{MF}_+$.  For any $M\in \text{MF}^*_+$ it follows from Theorem~\ref{FKT} that the subprocess $X^M$ also satisfies \text{Hypothesis \ref{HYP3}, \ref{HYP1}}  and \text{\ref{HYP5}} and its properly associated quasi-regular semi-Dirichlet form can be given by \eqref{EFMD2}. Replacing the Dirichlet form with semi-Dirichlet form in Definition~3.1 of \cite{JGY3}, we can similarly define the subordination of semi-Dirichlet forms.
Then we have the following two properties about the subordinations. Their proofs are completely the same as \text{Lemma 3.2} and \text{Corollary 3.3} of   \cite{JGY3}.

\begin{lemma}\label{LSS}
    \begin{description}
      \item[(1)] If $(\mathcal{E}^2,\mathcal{F}^2)$ is subordinate to $(\mathcal{E}^1,\mathcal{F}^1)$, then there exists an constant $C>0$ such that $\mathcal{E}^1_1(u,u)\leq C\cdot \mathcal{E}^2_1(u,u)$ for any $u\in \mathcal{F}^2$.
      \item[(2)] If $(\mathcal{E}^2,\mathcal{F}^2)$ is strongly subordinate to $(\mathcal{E}^1,\mathcal{F}^1)$, then any $\mathcal{E}^2$-nest is an $\mathcal{E}^1$-nest. Therefore any $\mathcal{E}^2$-q.c function is $\mathcal{E}^1$-q.c., and if $(\mathcal{E}^2,\mathcal{F}^2)$ is quasi-regular, then so is $(\mathcal{E}^1,\mathcal{F}^1)$.
    \end{description}
\end{lemma}

The following theorem is an analogy of Theorem~3.4 and 3.5 of  \cite{JGY3} which characterize the relationship between killing transform of Markov processes and subordination of Dirichlet forms.

\begin{theorem}\label{EMSS}
	Let $X$ and $(\mathcal{E},\mathcal{F})$ be given above. If $M\in \text{MF}^*_+$ and $(\mathcal{E}^M,\mathcal{F}^M)$ is the properly associated quasi-regular semi-Dirichlet form of the subprocess $X^M$, then $(\mathcal{E}^M,\mathcal{F}^M)$ is strongly subordinate to $(\mathcal{E},\mathcal{F})$. On the contrary assume $(\mathcal{E}',\mathcal{F}')$ to be another quasi-regular semi-Dirichlet form on $L^2(E,m)$ and $X'$ its associated Markov process.  If $(\mathcal{E}',\mathcal{F}')$ is strongly subordinate to $(\mathcal{E},\mathcal{F})$, then there exists an $M\in \text{MF}_+$ such that $X'$ is the subprocess of $X$ killed by $M$.
\end{theorem}

\section*{Acknowledgement}
The authors would like to thank the referees for their careful reading and many helpful comments on this paper.

\appendix
\renewcommand{\thesubsection}{\Alph{subsection}}

\section{Introduction to the (lower bounded) semi-Dirichlet forms}\label{SDF}

The definition of the semi-Dirichlet form is as follows.

		\begin{definition}\label{DEF5}
    Let $E$ be a metrizable Lusin space and $m$  a $\sigma$-finite positive measure on the Borel $\sigma$-algebra $\mathcal{B}$ of $E$. A bilinear form $(\mathcal{E},\mathcal{F})$ on $L^2(E,m)$ with $\mathcal{F}$ being dense in $L^2(E,m)$ is called a \emph{coercive closed form} if
    \begin{description}
      \item[($\mathcal{E}$1)] $(\tilde{\mathcal{E}},\mathcal{F})$ is positive definite and closed on $L^2(E,m)$, where
        \begin{equation}
            \tilde{\mathcal{E}}(u,v):= \frac{1}{2}(\mathcal{E}(u,v)+\mathcal{E}(v,u)),\quad  u,v\in \mathcal{F}
        \end{equation}
        is the \emph{symmetric part} of $\mathcal{E}$.
      \item[($\mathcal{E}$2)] (Sector condition) There exists a constant $K>0$ such that
       \begin{equation}
            \mathcal{E}_1(u,v)\leq K\mathcal{E}_1(u,u)^{\frac{1}{2}}\mathcal{E}_1(v,v)^{\frac{1}{2}},\quad u,v\in \mathcal{F}.
       \end{equation}
    \end{description}
Moreover $(\mathcal{E},\mathcal{F})$ is called a \emph{semi-Dirichlet form} on $L^2(E,m)$ if in addition:
   \begin{description}
     \item[($\mathcal{E}$3)] (Semi-Dirichlet property) For every $u\in \mathcal{F}$, $u^+\wedge 1\in \mathcal{F}$ and
     \[
        \mathcal{E}(u-u^+\wedge 1,u^+\wedge 1)\geq 0.
     \]
   \end{description}
\end{definition}

Note that the semi-Dirichlet property $(\mathcal{E}3)$ is in accordance with \cite{PJF} but contrary to \cite{MOR} and \cite{OY}. In fact in \cite{MOR} and \cite{OY} the semi-Dirichlet property means that for every $u\in \mathcal{F}$, $u^+\wedge 1\in \mathcal{F}$ and $\mathcal{E}(u^+\wedge 1,u-u^+\wedge 1)\geq 0$. In other words, the dual form $\hat{\mathcal{E}}(u,v):=\mathcal{E}(v,u)$ for any $u,v\in \mathcal{F}$ of the semi-Dirichlet form $(\mathcal{E,F})$ in Definition~\ref{DEF5} is a semi-Dirichlet form in the context of \cite{MOR} and \cite{OY}. Denote the \emph{antisymmetric part} of $\mathcal{E}$ by
\begin{equation}
    \check{\mathcal{E}}(u,v)=\frac{1}{2}(\mathcal{E}(u,v)-\mathcal{E}(v,u)), \quad  u,v\in \mathcal{F}.
\end{equation}
Obviously if $(\mathcal{E},\mathcal{F})$ is symmetric, then $\check{\mathcal{E}}=0$. The \emph{extended Dirichlet space} of $(\mathcal{E},\mathcal{F})$ is denoted by $\mathcal{F}_\text{e}$.
Let $(T_t)_{t\geq0},(G_\alpha)_{\alpha\geq 0}$ (resp. $(\hat{T}_t)_{t\geq0},(\hat{G}_\alpha)_{\alpha\geq0}$) denote the \emph{semigroup} and \emph{resolvent} (resp. \emph{co-semigroup} and \emph{co-resolvent}) of the semi-Dirichlet form $(\mathcal{E},\mathcal{F})$. In particular
\begin{equation}\label{EQUVM}
    (u,v)_m= \mathcal{E}_\alpha(v, G_\alpha u)=\mathcal{E}_\alpha(\hat{G}_\alpha u,v),\quad  u\in L^2(E,m),v\in \mathcal{F},\alpha>0.
\end{equation}
The semi-Dirichlet property $(\mathcal{E}3)$ is equivalent to  the Markov property:
if $0\leq u\leq 1$ and $u\in L^2(E,m)$, then $0\leq T_tu\leq 1$ for any $t\geq 0$ (equivalently $0\leq \alpha G_\alpha u\leq 1$ for any $\alpha\geq 0$).

We refer the quasi-notions of semi-Dirichlet forms, say 
\emph{$\mathcal{E}$-nest}, \emph{$\mathcal{E}$-exceptional} set, \emph{capacity} (denoted by Cap), \emph{$\mathcal{E}$-quasi-everywhere} ($\mathcal{E}$-q.e. in abbreviation), \emph{$\mathcal{E}$-quasi-continuous} ($\mathcal{E}$-q.c. in abbreviation),  \emph{quasi-regularity}, \emph{$m$-polar} and \emph{semipolar} etc, to \cite{BG}, \cite{PJF}, \cite{MR}, \cite{MOR} and \cite{OY}. Note that $N$ is $\mathcal{E}$-exceptional if and only if $\text{Cap}(N)=0$. The $\mathcal{E}$-q.c. $m$-version of $u$ is usually denoted by $\tilde{u}$. 
Every quasi-regular semi-Dirichlet form has a properly associated $m$-tight special standard process $X$, i.e. $P_tu(x):=E^x(u(X_t))$ is an $\mathcal{E}$-q.c. $m$-version of $T_tu$ for all $u\in L^2(E,m)$. Moreover if $(\mathcal{E,F})$ is quasi-regular then every function in $\mathcal{F}$ has an $\mathcal{E}$-q.c $m$-version.  Under the sector condition, a semipolar set is $m$-polar.

Fix a constant $\alpha\in [0,\infty)$. A positive function $u\in L^2(E,m)$ is called \emph{$\alpha$-excessive} (resp. \emph{$\alpha$-coexcessive}) if $e^{-\alpha t} T_tu\leq u$ (resp. $e^{-\alpha t} \hat{T}_tu\leq u$) for all $t>0$. Note that $u$ is $\alpha$-excessive (resp. $\alpha$-coexcessive) if and only if $\beta G_{\alpha+\beta}u\leq u$ (resp. $\beta \hat{G}_{\alpha+\beta}u\leq u$) for all $\beta>0$. In particular for any positive function $u\in L^2(E,m)$, $G_\alpha u$ is $\alpha$-excessive and $\hat{G}_\alpha u$ is $\alpha$-coexcessive with $\alpha>0$ ($\alpha\geq 0$ if $X$ is transient). A $0$-(co)excessive function is always called \emph{(co)excessive} in abbreviation. The following lemma will be used to prove Theorem~\ref{GCT}. We include the proof here for completion.

\begin{lemma}\label{EXC}
  Suppose that $u\in \mathcal{F}$ is $\alpha$-coexcessive and $v\in\mathcal{F}$. Then $u\wedge v\in \mathcal{F}$ and
\[
	\mathcal{E}_{\alpha}(u\wedge v,u\wedge v)\leq \mathcal{E}_{\alpha}(u\wedge v,v).
\]
In particular if $\alpha>0$ then $\mathcal{E}_\alpha(u\wedge v,u\wedge v)\leq K_\alpha^2\mathcal{E}_\alpha(v,v)$ for some constant $K_\alpha>0$.
\end{lemma}
\begin{proof}
Clearly $u\wedge v\in \mathcal{F}$ and $v=u\wedge v+ (v-u)_+$. It follows from the property of the approximating form that
\[\begin{aligned}
	\mathcal{E}_\alpha(u\wedge v,v-u\wedge v)&=\lim_{\beta\rightarrow \infty}\beta(u\wedge v-\beta \hat{G}_{\beta+\alpha}(u\wedge v),v-u\wedge v)_m\\
				&=\lim_{\beta\rightarrow \infty}\beta(u\wedge v-\beta \hat{G}_{\beta+\alpha}(u\wedge v),(v-u)_+)_m.
\end{aligned}\]
Note that $(u\wedge v)(x)(v-u)_+(x)= u(x)(v-u)_+(x)$ for any $x\in E$ and $\hat{G}_{\alpha+\beta}(u\wedge v)\leq \hat{G}_{\alpha+\beta}(u)$ because $\hat{G}_{\alpha+\beta}$ is positivity preserving. Since $u$ is $\alpha$-coexcessive we can deduce that $\beta \hat{G}_{\beta+\alpha}u\leq u$ and
\[
	\mathcal{E}_\alpha(u\wedge v,v-u\wedge v)\geq \lim_{\beta\rightarrow \infty}\beta(u-\beta \hat{G}_{\beta+\alpha}u,(v-u)_+)_m\geq 0.
\]
If $\alpha>0$ then it follows from the sector condition that  there exists a constant $K_\alpha>0$ such that
\[
	\mathcal{E}_{\alpha}(u\wedge v,u\wedge v)\leq \mathcal{E}_{\alpha}(u\wedge v,v)\leq K_\alpha \mathcal{E}_{\alpha}(u\wedge v,u\wedge v)^{\frac{1}{2}}\mathcal{E}_{\alpha}( v,v)^{\frac{1}{2}}.
\]
Therefore $\mathcal{E}_\alpha(u\wedge v,u\wedge v)\leq K_\alpha^2\mathcal{E}_\alpha(v,v)$.  
\end{proof}

The \emph{lower bounded semi-Dirichlet form} with a non-negative parameter $\alpha_0$ is a weaker form than the (non-negative) semi-Dirichlet form. Its definition is as follows.

\begin{definition}
A dense bilinear form $(\mathcal{E},\mathcal{F})$ on $L^2(E,m)$ is called the \emph{lower bounded closed form} if there exists a constant $\alpha_0\geq 0$ such that
\begin{description}
  \item[($\mathcal{E}\text{1}'$)] (Lower bounded) $\mathcal{E}_{\alpha_0}(u,u)\geq 0$ for any $u\in \mathcal{F}$ and $\mathcal{F}$ is  a Hilbert space with the norm $||\cdot ||_{\tilde{\mathcal{E}}_{\alpha}}$ for any $\alpha>\alpha_0$.
  \item[($\mathcal{E}\text{2}'$)] (Sector condition) There exists a constant $K>0$ such that
\[
	\mathcal{E}_{\alpha_0}(u,v)\leq K\mathcal{E}_{\alpha_0}(u,u)^{\frac{1}{2}}\mathcal{E}_{\alpha_0}(v,v)^{\frac{1}{2}}.
\]
\end{description}
Moreover $(\mathcal{E,F})$ is called the \emph{lower bounded semi-Dirichlet form with the parameter} $\alpha_0$ if in addition $(\mathcal{E,F})$ also satisfies the semi-Dirichlet property ($\mathcal{E}\text{3}$) in Definition~\ref{DEF5}.
\end{definition}

For any lower bounded closed form $(\mathcal{E},\mathcal{F})$ there exist two unique strongly continuous semigroups (not necessarily to be contractive) $(T_t)_{t\geq 0}, (\hat{T}_t)_{t\geq 0}$ on $L^2(E,m)$ such that $||T_t||\leq \text{e}^{\alpha_0 t}, ||\hat{T}_t||\leq \text{e}^{\alpha_0 t}$ and similarly their corresponding resolvents satisfy \eqref{EQUVM} for any $f\in L^2(E,m),u\in \mathcal{F}$ and $\alpha>\alpha_0$.
 Moreover define the approximating form $\mathcal{E}^\alpha$ by
\[
	\mathcal{E}^\alpha(u,v):=\alpha(u, v-\alpha G_\alpha v)_m,\quad u,v\in L^2(E,m),
\]
then $u\in \mathcal{F}$ if and only if $\overline{\lim}_{\alpha\rightarrow \infty}\mathcal{E}^\alpha (u,u)<\infty$ and if $u,v\in \mathcal{F}$, then $\lim_{\alpha\rightarrow \infty} \mathcal{E}^\alpha (u,v)=\mathcal{E}(u,v)$. The following lemma is used to prove Theorem~\ref{FKT} and its proof is obvious.


\begin{lemma}\label{LBTN}
	Let $(\mathcal{E},\mathcal{F})$ be a lower bounded closed form and $T_t, \hat{T}_t,G_\alpha,\hat{G}_\alpha$ the associated strongly continuous semigroups and resolvents. Then $(\mathcal{E},\mathcal{F})$ is non-negative if and only if the semigoup $(T_t)_{t\geq 0}$ (or equivalently the resolvent $(G_\alpha)_{\alpha\geq 0}$) is contractive, i.e. $||T_t||\leq 1$ for any $t\geq 0$ (or equivalently $||\alpha G_\alpha ||\leq 1$ for any $\alpha\geq 0$).
\end{lemma}

\section{The correspondence between the PCAFs and smooth measures}\label{PCAF}

We refer the definition of  the  additive functionals (AFs) of $X$ to Definition~3.16 of  \cite{PJF}.
Note that two AFs $A$ and $B$ are $m$-equivalent if and only if $P^m(A_t\neq B_t)=0$ for all $t>0$.

 For any PCAF $(A_t)_{t\geq0}$, there exists a unique $\sigma$-finite measure $\mu_A$ on $E$ charging no $\mathcal{E}$-exceptional sets such that
\begin{equation}\label{RFC}
    (f,\tilde{\hat{h}})_{\mu_A}=\lim_{t\rightarrow 0}\frac{1}{t}E^{\hat{h}\cdot m}\int_0^tf(X_s)dA_s=\lim_{\beta\rightarrow \infty}\beta(\hat{h},U^\beta_Af)_m
\end{equation}
for any non-negative function $f$ and $\alpha$-coexcessive function $\hat{h}\in \mathcal{F}$ with $\alpha\geq 0$. The sequences appeared in \eqref{RFC} are increasing relative to $t\downarrow 0$ or $\beta\uparrow \infty$. Here
\begin{equation}
     U^\beta_Af(x):=E^x\int_0^\infty e^{-\beta t}f(X_t)dA_t,\quad x\in E.
\end{equation}
The unique $\sigma$-finite measure $\mu_A$ relative to $(A_t)_{t\geq 0}$ is also called the \emph{Revuz measure} of the PCAF $(A_t)_{t\geq 0}$.
On the contrary a measure $\mu$ on $E$ charging no $\mathcal{E}$-exceptional sets is the Revuz measure of a PCAF of $X$ if and only if one of the following equivalent conditions holds (see Theorem~4.22 of  \cite{PJF}):
 \begin{description}
   \item[(1)] There is a q.c. function $f$ such that  $f>0$ q.e. and $\mu(f)<\infty$.
   \item[(2)] There is an $\mathcal{E}$-nest $\{K_n:n\geq 1\}$ of compact subsets of $E$ such that $\mu(K_n)<\infty$ for any $n\geq 1$.
 \end{description}
 A positive measure $\mu$ on $(E,\mathcal{B}(E))$ is called \emph{smooth with respect to $(\mathcal{E},\mathcal{F})$}, denoted by $\mu\in S$, if $\mu(N)=0$ for any $\mathcal{E}$-exceptional set $N\in \mathcal{B}(E)$ and $\mu$ satisfies any of the above two conditions.
The PCAFs of $X$ and the smooth measures of $(\mathcal{E},\mathcal{F})$ have a one-to-one correspondence (up to the equivalence of PCAFs) by the formula \eqref{RFC}. We refer more equivalent conditions of \eqref{RFC} to Theorem~A.8 of  \cite{MMS}.
A characterization to the Revuz measure by P.J.Fitzsimmons in  \cite{PJF} is very useful to prove \text{Lemma \ref{UFV}}.

\begin{lemma}[Corollary 4.16,  \cite{PJF}]\label{RFP}
    Let $A$ be any PCAF of $X$ and $\mu_A$ the associated Revuz measure of $A$. Then
    \begin{equation}
        \mathcal{E}_\alpha(h,U_A^\alpha f)=\mu_A(\tilde{h}f)
    \end{equation}
    for any $\alpha>0$, all $h\in \mathcal{F}$ and all $f\in \mathcal{B}^+$ for which $U_A^\alpha f\in \mathcal{F}$. If $X$ is transient,
    \[
         \mathcal{E}(h,U_A f)=\mu_A(\tilde{h}f)
    \]
    for all $h\in \mathcal{F}_\text{e}$ and all $f\in \mathcal{B}^+$ for which $U_A f\in \mathcal{F}_\text{e}$.
\end{lemma}

\begin{remark}\label{IFW}
    In fact we can remove the condition $\hat{h}\in\mathcal{F}$ in \eqref{RFC}, in other words, \eqref{RFC} holds for any non-negative function $f$ and $\alpha$-coexcessive function $\hat{h}$ with $\alpha>0$ ($\alpha= 0$ if $X$ is transient).
 To see this, choose an  $\alpha$-coexcessive strictly positive q.c. function $\hat{g}\in \mathcal{F}$ (we refer its existence to Theorem 2.4.8 of  \cite{OY} for $\alpha>0$ and Proposition 3.3 of  \cite{BLNB} for $\alpha=0$ if $X$ is transient). Set $h_n:=\hat{h}\wedge n\hat{g}\in \mathcal{F}$ and $h_n$ is $\alpha$-coexcessive by Lemma~1.4.2 of \cite{OY}. Then $h_n$ has a q.c. $m$-version $\tilde{h}_n$. Define a quasi-open set
 \[
 G_n:=\{x: \tilde{h}_n(x)<n\hat{g}(x) \}.
 \]
  It follows that $\cup_{n=1}^\infty G_n=E$ q.e. and hence the function $\tilde{\hat{h}}$ defined by
    $
        \tilde{\hat{h}}(x):=\tilde{h}_n(x)  
   $ for any $x\in G_n,  n\geq 1$
    is clearly a q.c. $m$-version of $\hat{h}$. Moreover $h_n\uparrow \hat{h}$ $m$-a.e. and $\tilde{h}_n \uparrow \tilde{\hat{h}}$ q.e. Since \eqref{RFC} holds for every $h_n$, by letting $n\rightarrow \infty$, it also holds for $\hat{h}$ and its q.c. $m$-version $\tilde{\hat{h}}$.
\end{remark}

\section{Multiplicative functionals and the killing transforms}\label{DEFMF}

We refer the definition of the multiplicative functionals of $X$ to  \cite{JGY} and \cite{JGY3}.
Note that all the equalities or inequalities about MFs and AFs appeared in this section are in the sense of $P^m$-a.s.
Let $\text{MF}(X)$ (or MF if $X$ is fixed) be the set of all exact multiplicative functionals of $X$. It is convenient to suppose that $M_t=0$ for $t\geq \zeta$. Two MFs $M,N\in \text{MF}(X)$ are \emph{$m$-equivalent} provided that for each $t>0$, $M_t=N_t\; P^m$-a.s. on $\{\zeta>t\} $. For any $M\in$MF, write
\[
    S_M:=\inf\{t>0:M_t=0\},\quad E_M:=\{x\in E: P^x(M_0=1)=1  \},
\]
for the \emph{life time} and the set of \emph{permanent points} of $M$. Clearly $E_M$ is also the set of all irregular points of $S_M$, i.e. $E_M=\{x\in E: P^x(S_M>0)=1 \}$. In particular $E_M$ is a finely open set. If $E_M$ is nearly optional, $M$ is called a \emph{right MF}. Further let
\[
    \text{MF}_+:=\{M\in \text{MF}: S_M>0\;P^m\text{-a.s.} \},
\]
\[
     \text{MF}_{++}:=\{M\in \text{MF}:M \text{ does not vanish, i.e. } S_M\geq \zeta\;P^m\text{ -a.s.}\}.
\]
Then $\text{MF}_{++}\subset \text{MF}_+\subset \text{MF}$.  If $M\in \text{MF}_+$, then $E_M=E\;m$-a.e., whereas $E_M$ is finely open (hence q.e. quasi-open). It follows that $E_M=E$ q.e.

If $M$ is right, define for each $x\in E_M$ a probability $Q^x$ on $(\Omega, \mathcal{M})$ by
\begin{equation}\label{POS}
	Q^x(Z):=E^x\int_0^\infty Z\circ k_t d(-M_t),\quad Z\in b\mathcal{M}
\end{equation}
where $(k_t)_{t\geq 0}$ are the killing operators on $\Omega$ defined by $k_t\omega(s)=\omega(s)$ if $t>s$ and $k_t\omega(s)=\Delta$ if $t\leq s$.
Then $(\Omega, \mathcal{M}, (\mathcal{M}_t), X_t, \theta_t, Q^x)$ is also a right Markov process with the state space $E_M$ and lifetime $S_M$, which is called the \emph{$M$-subprocess} of $X$ (or the subprocess of $X$ killed by $M$) and denoted by $(X,M)$ or $X^M$. The semigroup $(Q_t)_{t\geq 0}$ and resolvent $(V^\alpha)_{\alpha\geq 0}$ of $X^M$ can be given by
\begin{equation}
\begin{aligned}
  Q_t f(x) & = E^x(f(X_t)M_t), \\
   V^\alpha f(x) &= E^x\int_0^\infty e^{-\alpha t}f(X_t)M_tdt
\end{aligned}
\end{equation}
for $x\in E_M, t\geq 0, \alpha\geq 0$ and $Q_t(x,\cdot)=V^\alpha(x,\cdot)=0$ for any $x\notin E_M, t\geq 0, \alpha\geq 0$. An $(\mathcal{M}_t)$-stopping time $T$ is called a \emph{terminal time} if $T=t+T\circ \theta_t$ identically on $\{ t<T \}$. If $T$ is a terminal time, $1_{[0,T)}(t)$ is an MF of $X$ and $S_M$ is a terminal time if $M\in \text{MF}(X)$.
Let Exc$^\alpha$ (resp. Exc$^\alpha$($M$)) denote all of the $\alpha$-excessive functions of $X$ (resp. $X^M$) and in particular $\alpha$ will be omitted if it equals $0$. Clearly, $\text{Exc}^\alpha\subset \text{Exc}^\alpha(M)$.

Fix an $M\in\text{MF}$. We also refer the definition of the $M$-additive functionals of $X$ to  \cite{JGY}.
Let $\text{AF}(X,M)$ or $\text{AF}(M)$ (resp. $\text{PCAF}(X,M)$ or $\text{PCAF}(M)$) denote the set of all $M$-(resp. continuous) additive functionals. If $M_t\equiv 1$, then CAF(1) is exactly the set of all the PCAFs of $X$ introduced in Appendix~\ref{PCAF}. We write PCAF for CAF(1). For a terminal time $T$, $\text{AF}(T):=\text{AF}(1_{[0,T)})$ and write $T$-additive functional for $1_{[0,T)}$-additive functional. Further let $\bar{M}_t:=1-M_t$. Clearly $\bar{M}\in \text{AF}(M)$. Moreover the Stieltjes logarithm of $M$
\begin{equation}
    (\text{slog}M)_t :=\int_0^t 1_{\{s<S_M\}}\frac{d(-M_s)}{M_{s-}},\quad t\geq 0
\end{equation}
is an $S_M$-additive functional. We usually write $([M]_t)_{t\geq 0}$ for $((\text{slog}M)_t)_{t\geq 0}$.

\section{Proof of Theorem~\ref{GCT}}\label{PROOF}

\begin{proof}
    Similarly to the discussions to the killing transforms in Theorem~4.1 of  \cite{JGY} it follows from Theorem~5.10 of  \cite{PJF}, Proposition~\ref{VMS} and Theorem~\ref{FKT} that we only need to deal with the case of $M_t=1_{[0,S_M)}(t),t\geq 0$ such that $S_M$ is a terminal time and $S_M>0$ a.s.  In particular there exists  a subset $B\subset E\times E\setminus d$ such that
    $S_M=J_B:=\inf\{t>0:(X_{t-},X_t)\in B\}.$
    Without loss of generality assume that $E_M=E$ and \eqref{EUBM} holds for $\nu_{\bar{M}}=\nu_{S_M}=1_B\cdot \nu$. Further set $V^1f(x)=E^x\int_0^{S_M} f(X_s)ds$, similarly to Theorem~\ref{FKT} it suffices  to prove that for $f\in pL^2(E,m)$ and $u\in p\mathcal{F}$,
    \begin{equation}\label{VFUF}
        V^1f\in \mathcal{F},\quad (u,f)_m=\mathcal{E}(u,V^1f)+\nu_{S_M}(u\otimes V^1f).
    \end{equation}
First assume $B\subset \{(x,y):\rho(x,y)>c\}$ for some  constant $c>0$ where $\rho$ is the metric on $E$. Then \eqref{VFUF} can be proved for this case similarly to   Theorem~4.1 of  \cite{JGY}.
For general $B\subset E\times E\setminus d$ let
    \[
        B_n:=B\cap \{(x,y):\rho(x,y)>\frac{1}{n}\},\quad T_n:=J_{B_n}.
    \]
    Since $S_M>0$ a.s. it follows that $\{T_n\}$ well converges decreasingly to $S_M$ in the sense that for any $\omega\in \Omega$, there exists  a constant $N=N(\omega)$ such that $T_n(\omega)=S_M(\omega)$ for all $n>N$. Denote by $(V^q_n)$ the resolvent of $(X,T_n)$ and then
    \begin{equation}\label{VNF}
        V^1_nf(x)=E^x\int_0^{T_n}\text{e}^{-t}f(X_t)dt \downarrow E^x\int_0^{S_M}\text{e}^{-t}f(X_t)dt=V^1f(x)
    \end{equation}
    as $n\rightarrow \infty$ for any $x\in E$. Clearly we have
    \[
        V^1_n f\in \mathcal{F},\quad (u,f)_m=\mathcal{E}_1(u,V^1_nf)+\nu_{T_n}(u\otimes V^1_nf)
    \]
    for $f\in pL^2(E,m)$ and $u\in p\mathcal{F}$.
Similarly to Theorem~4.1 of  \cite{JGY} that $V^1_nf\rightarrow V^1f$ weakly in $\mathcal{F}$ and in particular $V^1f\in \mathcal{F}$. Thus we need to prove for $f\in pL^2(E,m)$ and $u\in p\mathcal{F}$ that
    \[
        \lim_{n\rightarrow \infty}\nu_{T_n}(u\otimes V^1_nf)=\nu_{S_M}(u\otimes V^1f).
    \]
    Without loss of generality assume $X$ to be transient (see \text{\S\ref{THWD}}). The notations $\hat{g}, \hat{m}=\hat{g}\cdot m$ and $\check{X}$ are given in the notes before \text{Lemma \ref{DECOM}}. In particular $X$ and $\check{X}$ are in duality relative to the excessive function $\hat{m}$. Note that $\hat{g}\in \mathcal{F}$ is a q.e. strictly positive coexcessive q.c. function. Let $\check{B}=\{(x,y):(y,x)\in B\}$, $\check{S}_M=\check{J}_{\check{B}}:=\inf\{t>0:(\check{X}_{t-},\check{X}_t)\in \check{B}\}$ and similarly $\check{T}_n:=\check{J}_{\check{B}_n}$. Then $\check{S}_M$ (resp. $\check{T}_n$) is dual to $S_M$ (resp. $T_n$) relative to $\hat{m}$ and $\check{\nu}_{\check{S}_M}^{\hat{m}}$ (resp. $\check{\nu}_{\check{T}_n}^{\hat{m}}$) is the dual bivariate Revuz measure of $\nu_{S_M}^{\hat{m}}$ (resp.  $\nu_{T_n}^{\hat{m}}$) relative to $\hat{m}$ (see \S6 of  \cite{JGY}). Clearly $\check{T}_n$ well converges to $\check{S}_M$. It follows from (I.5.10, I.3.6) of  \cite{JGY} and \text{Theorem~\ref{EXIS}} that
    \begin{equation}\label{NTUO}
    \begin{aligned}
        \nu_{T_n}(u\otimes V^1_nf)&=(\hat{g}\cdot \nu_{T_n})((u/\hat{g})\otimes V^1_nf)=\nu_{T_n}^{\hat{m}}((u/\hat{g})\otimes V^1_nf)\\&= \check{\nu}_{\check{T}_n}^{\hat{m}}(V^1_nf\otimes(u/\hat{g}))=(f,\check{P}^1_{\check{T}_n}(u/\hat{g}))_{\hat{m}}
    \end{aligned}\end{equation}
    where $\check{P}^1_{\check{T}_n}(u/\hat{g})(x):=\check{P}^x[\text{e}^{-\check{T}_n}(u/\hat{g})(\check{X}_{\check{T}_n})],x\in E$. From the well convergence of $\{\check{T}_n\}$ we can deduce that \[\check{P}^1_{\check{T}_n}(u/\hat{g})(x)\rightarrow \check{P}^1_{\check{S}_M}(u/\hat{g})(x)\] pointwisely. Let $w_k:=u\wedge k\hat{g}\in \mathcal{F}$ for any $k\geq 1$. It follows from Lemma~\ref{EXC} that
    \[
    	\mathcal{E}(w_k.w_k)\leq \mathcal{E}(w_k,u)\leq K_1\mathcal{E}_1(w_k,w_k)^{\frac{1}{2}}\cdot \mathcal{E}_1(u,u)^{\frac{1}{2}}
    	\]
  for some constant $K_1>0$. Hence
    \begin{equation}\label{MEWK}
        \mathcal{E}_1(w_k,w_k)\leq (K_1+1)^2\mathcal{E}_1(u,u),\quad  k\geq 1
    \end{equation}
and it follows from \eqref{NTUO} that
    \[
        \nu_{T_n}(w_k\otimes V^1_nf)=(f,\check{P}^1_{\check{T}_n}(w_k/\hat{g}))_{\hat{m}}.
    \]
On the other hand  since $f\in pL^2(E,m)$ we have
  \[
  \int f(x)\hat{m}(dx)=\int f(x)\hat{g}(x)m(dx)<\infty.
  \]
   From $|w_k|\leq k\hat{g}$ and the bounded convergence theorem we can deduce that
    \[
        \nu_{T_n}(w_k\otimes V^1_nf)=(f,\check{P}^1_{\check{T}_n}(w_k/\hat{g}))_{\hat{m}}\rightarrow (f,\check{P}^1_{\check{S}_M}(w_k/\hat{g}))_{\hat{m}}
    \]
    as $n\rightarrow \infty$. Similarly to \eqref{NTUO} it follows that
    \[
        (f,\check{P}^1_{\check{S}_M}(w_k/\hat{g}))_{\hat{m}}=\nu_{S_M}(w_k\otimes V^1f)
    \]
and thus $(w_k,f)_m=\mathcal{E}_1(w_k,V^1f)+\nu_{S_M}(w_k\otimes V^1f)$. Since $w_k\uparrow u$ and $w_k$ is weak-$\tilde{\mathcal{E}}_1$ convergent to $u$ by \eqref{MEWK} (hence a C\'{e}saro average subsequence of $\{w_k\}$ is  strongly convergent to $u$), it follows from the monotone convergence theorem that
    \[
    (u,f)_m=\mathcal{E}_1(u,V^1f)+\nu_{S_M}(u\otimes V^1f).
    \]
That completes the proof.     
\end{proof}



\end{document}